\documentclass[12pt,oneside]{amsart}
\usepackage{amssymb, amsmath, amsthm}

\usepackage{epsfig}
\usepackage{epstopdf}
\usepackage{graphicx}
\usepackage{color}
\usepackage{mathptmx}
\usepackage{enumerate}

\usepackage{pinlabel}

\theoremstyle{plain}
\newtheorem{theorem}{Theorem}[section]

\newtheorem*{theorem*}{Theorem}
\newtheorem*{theorem-Main A}{Theorem \ref{Thm: Main A}}
\newtheorem*{theorem-Main B}{Theorem \ref{Thm: Main Theorem B} (rephrased)}
\newtheorem*{theorem-Main C}{Theorem \ref{Thm: Main Theorem C} (rephrased)}
\newtheorem*{theorem-Tunnel 1}{Theorem \ref{Thm: Tunnel 1}}
\newtheorem*{theorem-Tunnel 2}{Theorem \ref{Thm: Tunnel 2}}

\newtheorem{corollary}[theorem]{Corollary}
\newtheorem{lemma}[theorem]{Lemma}

\theoremstyle{definition}

\newtheorem*{remark}{Remark}

\newcommand{\Z}{\mathbb Z}
\newcommand{\N}{\mathbb N}

\newcommand{\nil}{\varnothing}

\newcommand{\wihat}{\widehat}
\newcommand{\defn}[1]{\textbf{#1}}
\newcommand{\boundary}{\partial}

\newcommand{\mc}[1]{\mathcal{#1}}
\newcommand{\ob}[1]{\overline{#1}}
\newcommand{\inter}[1]{\mathring{#1}}

      \makeatletter
      \def\@setcopyright{}
      \def\serieslogo@{}
      \makeatother
      
\begin{document}

   \title[Band Taut Sutured Manifolds]{Band Taut Sutured Manifolds}
   \author{Scott A Taylor}
   \email{sataylor@colby.edu}
   \thanks{}

  \begin{abstract}
Attaching a 2--handle to a genus two or greater boundary component of a 3--manifold is a natural generalization of Dehn filling a torus boundary component. We prove that there is an interesting relationship between an essential surface in a sutured 3--manifold, the number of intersections between the boundary of the surface and one of the sutures, and the cocore of the 2--handle in the manifold after attaching a 2--handle along the suture. We use this result to show that tunnels for tunnel number one knots or links in any 3--manifold can be isotoped to lie on a branched surface corresponding to a certain taut sutured manifold hierarchy of the knot or link exterior. In a subsequent paper, we use the theorem to prove that  band sums satisfy the cabling conjecture, and to give new proofs that unknotting number one knots are prime and that genus is superadditive under band sum. To prove the theorem, we introduce band taut sutured manifolds and prove the existence of band taut sutured manifold hierarchies.
\end{abstract}
  
    \maketitle
  \section{Introduction}
  
Gabai's sutured manifold theory \cite{G1, G2, G3} is central to a number of stunning results concerning Dehn surgery on knots in 3--manifolds. Many of these insights make use of a famous theorem of Gabai \cite[Corollary 2.4]{G2}: with certain mild hypotheses, there is at most one way to fill a torus boundary component of a 3--manifold so that Thurston norm decreases. Lackenby \cite{L1}, building on this work, proved a theorem relating Dehn surgery properties of a knot to the intersection between the knot and essential surfaces in the 3--manifold. Lackenby used his results to study the effect of twisting the unknot along a knot having linking number zero with the unknot, and to study \cite{L2} the uniqueness properties of Dehn surgery on certain knots in certain 3-manifolds. Lackenby \cite{L3} and Kalfagianni \cite{K} also used Lackenby's theorem to study the unknotting properties of certain knots. 

In this paper, we prove a version of Lackenby's theorem for attaching a 2--handle to a sutured 3--manifold along a suture. Like Lackenby, we use Scharlemann's combinatorial version of sutured manifold theory \cite{MSc3}. Although our method is inspired by the proofs of Gabai's and Lackenby's theorems, the proof is very different.

For the statement of the main result, let $(N,\gamma)$ be a sutured manifold and let $F \subset \boundary N$ be a component of genus at least 2. Let $b \subset \gamma \cap F$ be a component. Let $N[b]$ be the 3--manifold obtained by attaching a 2--handle to $N$ along $b$ and let $\beta \subset N[b]$ be the cocore of the 2--handle. For a 3-manifold $M$ with $T$ the union of two torus boundary components, we say that a non-zero homology class $y \in H_2(M,\boundary M)$ is \defn{Seifert-like} for $T$ if the projection of $y$ to the first homology of each component of $T$ is non-zero. The main result of this paper is:

\begin{theorem-Main C}
Suppose that $(N,\gamma)$ is taut and that the components of $\boundary N - A(\gamma)$ adjacent to $b$ are both thrice punctured spheres or are both once-punctured tori. Let $Q \subset N$ be a surface having no component a sphere or disc disjoint from $\gamma$. Assume that $\boundary Q$ intersects $\gamma$ minimally and that $|\boundary Q \cap b| \geq 1$. Then one of the following is true:
\begin{enumerate}
\item $Q$ has a compressing or $b$-boundary compressing disc.
\item $(N[b],\beta) = (M'_0,\beta'_0) \# (M'_1,\beta'_1)$ where $M'_1$ is a lens space and $\beta'_1$ is a core of a genus one Heegaard splitting of $M'_1$.
\item The sutured manifold $(N[b],\gamma-b)$ is taut. The arc $\beta$ can be properly isotoped  to be embedded on a branched surface $B(\mc{H})$ associated to a taut sutured manifold hierarchy $\mc{H}$ for $N[b]$. There is also a proper isotopy of $\beta$ in $N[b]$ to an arc disjoint from the first decomposing surface in $\mc{H}$. If $b$ is adjacent to thrice-punctured spheres, that first decomposing surface can be taken to represent $\pm y$ for any given non-zero $y \in H_2(N[b],\boundary N[b])$. If $b$ is adjacent to once-punctured tori, the first decomposing surface can be taken to represent any Seifert-like homology class for the corresponding unpunctured torus components of $\boundary N[b]$. 
\item 
\[
-2\chi(Q) + |\boundary Q \cap \gamma| \geq 2|\boundary Q \cap b|.
\]
\end{enumerate}
\end{theorem-Main C}

A \defn{$b$-boundary compressing disc} for a properly embedded surface $Q \subset N$ transverse to $b$ is a disc with boundary consisting of an arc on $Q$ and a subarc of $b$ and with interior disjoint from $Q \cup \boundary N$. See Figure \ref{Fig: b-boundary compressing disc}.

\begin{figure}[ht!] 
\labellist \small\hair 2pt 
\pinlabel {$F \subset \boundary N$} [br] at 582 4
\pinlabel {$b$} [br] at 556 76
\pinlabel {$Q$} [bl] at 398 110
\pinlabel {$D$} [l] at 324 142
\endlabellist 
\centering 
\includegraphics[scale=.4]{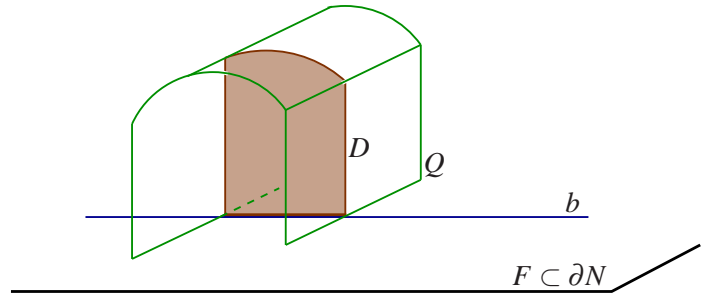}
\caption{$D$ is an $b$-boundary compressing disc for the surface $Q$ outlined in green.}
\label{Fig: b-boundary compressing disc}
\end{figure}

In \cite{L2}, Lackenby shows how to add sutures to the (non-empty) boundary of a compact, orientable, irreducible and boundary-irreducible manifold (other than a 3--ball) to create a taut sutured manifold. In his construction all components of $R(\gamma)$ are thrice-punctured spheres or tori, so the hypothesis in Theorem \ref{Thm: Main Theorem C} that $b$ be adjacent to thrice-punctured sphere components of $R(\gamma)$ is reasonable. 

The fourth conclusion is useful, since the inequality can be rearranged to be:
\[
-2\chi(Q) + |\boundary Q \cap (\gamma - b)| \geq |\boundary Q \cap b|.
\]
Thus, for example, if $\boundary Q$ happens to be disjoint from $\gamma - b$, then twice the negative euler characteristic of the surface is an upper bound for the number of times the boundary of the surface intersects $b$. 

The third conclusion of the theorem is of particular interest in that it is related to several well-known and very useful facts:
\begin{itemize}
\item If $K$ is an unknotting number one knot in $S^3$ and if $\beta$ is an arc in the knot complement defining a crossing change converting $K$ into the unknot then $\beta$ is isotopic into a minimal genus Seifert surface for $K$.
\item If $K$ is a tunnel number one knot in $S^3$ and if $\beta$ is a tunnel then, if the Scharlemann-Thompson invariant \cite{ST} is not 1, $\beta$ can be isotoped into a minimal genus Seifert surface for $K$.
\end{itemize}

Any minimal genus Seifert surface can be used as the first surface in a taut sutured manifold hierarchy of the knot exterior, and so any minimal genus Seifert surface can be thought of as part of a branched surface associated to a taut sutured manifold hierarchy of the knot exterior. Since these facts have proven to be very useful, the third conclusion of the main theorem of this paper also has the potential to be useful and perhaps points to a connection between the various \textit{ad hoc} methods used to push certain arcs onto minimal genus Seifert surfaces.

Applications of Theorem \ref{Thm: Main Theorem C} include a proof that knots that are band sums satisfy the cabling conjecture \cite[Theorem 8.1]{T3}, a partial solution to a conjecture of Scharlemann and Wu \cite[Corollary 5.4]{T3}, a near complete solution of a conjecture of Scharlemann \cite[Corollary 6.2]{T3}, and new proofs of three classical facts:
\begin{itemize}
\item Knot genus is superadditive under band connect sum \cite[Theorem 7.3]{T3}.
\item Unknotting number one knots are prime \cite[Theorem 7.2]{T3}.
\item Tunnel number one knots in $S^3$ have minimal genus Seifert surfaces disjoint from a given tunnel (Theorem \ref{ST-disjt tunnel} below).
\end{itemize}

These three facts previously all had proofs which use sutured manifold theory, but the methods were different. The advantage of the new proofs is that they are all nearly identical. Since some effort is required to rephrase the theorems in a way in which the main theorem of this paper can be usefully applied, we defer proofs for all but the last fact to \cite{T3}; the new proof of the last fact is given in this paper. Indeed, we prove the following stronger theorem for tunnel number one knots and 2--component links in any 3--manifold admitting such a knot or link (see Section \ref{Tunnels} for the definitions):

\begin{theorem-Tunnel 1}
Suppose that $L_b$ is a knot or 2-component link in a closed, orientable 3-manifold $M$ such that $L_b$ has tunnel number one. Let $\beta$ be a tunnel for $L_b$. Assume also that $(M - L_b, \beta)$ does not have a (lens space, core) summand. Then there exist  (possibly empty) curves $\wihat{\gamma}$ on $\boundary(M - \inter{\eta}(L_b))$ such that $(M - \inter{\eta}(L_b), \wihat{\gamma})$ is a taut sutured manifold and the arc $\beta$ can be properly isotoped to lie on the branched surface associated to a taut sutured manifold hierarchy of $(M - \inter{\eta}(L_b), \wihat{\gamma})$. In particular, if $L_b$ has a (generalized) Seifert surface, then there exists a minimal genus (generalized) Seifert surface for $L_b$ that is disjoint from $\beta$.
\end{theorem-Tunnel 1}

\section{Motivation and Outline}
As motivation for our proof of Theorem \ref{Thm: Main Theorem C}, we briefly review the proofs of Gabai's and Lackenby's theorems. For reference, here are (simplified and weakened) statements of Gabai's and Lackenby's theorems. The sutured manifold terminology will be explained in the next section.

\begin{theorem*}[Gabai]
Let $N$ be an atoroidal Haken 3--manifold whose boundary is the non-empty union of tori. Let $S$ be a Thurston norm minimizing surface representing an element of $H_2(N,\boundary N)$ and let $P$ be a component of $\boundary M$ such that $P \cap S = \nil$. Then, with at most one exception, $S$ remains norm minimizing in each manifold obtained by Dehn filling $N$ along a slope in $P$.
\end{theorem*}

\begin{theorem*}[{\cite[Theorem 1.4]{L1}}]
Suppose that $(N,\gamma)$ is a taut atoroidal sutured manifold. Let 
$P \subset \boundary N$ be a torus component disjoint from $\gamma$.

Suppose that $b$ is a slope on $P$ such that Dehn filling $N$ with slope $b$ creates a sutured manifold $(N(b),\gamma)$ that is not taut. Let $Q \subset N$ be an essential surface such that $\boundary Q$ intersects $b$ minimally and $|\boundary Q \cap b| \geq 1$. Then 
\[
 -2\chi(Q) + |\boundary Q \cap \gamma| \geq 2|\boundary Q \cap b|.
\]
\end{theorem*}

Gabai's theorem is proved by taking a taut sutured manifold hierarchy for $N$ such that each decomposing surface in the hierarchy is disjoint from $P$. The first decomposing surface is the given surface $S$. The hierarchy ends at a taut sutured manifold $(N_n,\gamma_n)$ such that $H_2(N_n,\boundary N_n - P) = 0$. Our additional assumption that $N$ is atoroidal implies that $N_n$ consists of 3--balls and one additional component that is homeomorphic to $P \times [0,1]$. Dehn filling $N_n$ along a slope $b \subset P$ creates another sutured manifold which we call $(N_n(b),\gamma_n)$. An examination of the sutures $\gamma_n \cap \boundary N_n$ shows that for all but at most one choice of $b$, $(N_n(b),\gamma_n)$ remains taut. One of the fundamental theorems of sutured manifold theory \cite[Corollary 3.9]{MSc3} (see Theorem \ref{Thm: Tautness Up}
below) implies that, except for the exceptional slope, $(N(b),\gamma)$ and $S$ are taut. 

Equivalently, we can begin with the Dehn-filled $\beta$-taut sutured manifold $(N(b),\gamma,\beta)$ where $\beta$ is the core of the surgery solid torus. The hierarchy for $N$ is then a $\beta$-taut sutured manifold hierarchy for $(N(b),\gamma,\beta)$ where each decomposing surface is disjoint from $\beta$. We conclude that for all but at most one choice of $b$, the sutured manifold $(N(b),\gamma,\nil)$ and the surface $S$ are not only $\beta$--taut, but also $\nil$--taut. There are two advantages to this viewpoint. One is that it is possible to see that if the hierarchy is taut then $\beta$ has infinite order in the fundamental group of $N(b)$ \cite[Theorem A.6]{L2}. The other advantage is that, if the hierarchy of the filled manifold is taut, it is not difficult to see that $\beta$ can be isotoped to lie on the branched surface corresponding to the hierarchy. The analogous statment in Theorem \ref{Thm: Main Theorem C} is much harder to prove.

That was a sketch of the proof of Gabai's theorem. We now turn to Lackenby's theorem. The surface $Q$ in the statement of Lackenby's theorem is an example of what is called a ``parameterizing surface'' in $(N, \gamma)$. (Parameterizing surfaces are defined in Section \ref{Param}.) Associated to each parameterizing surface is a number called the index (or ``sutured manifold norm'' \cite{CC}). In the case of Lackenby's theorem, the index of $Q$ is defined to be:
\[
I(Q) = |\boundary Q \cap \gamma| -2\chi(Q).
\]

Suppose now that $b \subset P$ is the exceptional slope, so that $(N(b),\gamma-b)$ is not taut. Let $Q \subset N$ be a parameterizing surface so that $\boundary Q$ intersects $\gamma$ minimally and $|\boundary Q \cap b| > 0$.

In the sutured manifold $(N_n,\gamma_n)$, the surface $Q$ has decomposed into a parameterizing surface $Q_n$ with $I(Q) \geq I(Q_n)$. The component $N'_n$ of $N_n$ containing $P$ is homeomorphic to $P \times [0,1]$. Some components of $Q_n$ lie in $N'_n$. Since $Q$ is essential, $Q_n \cap N'_n$ is the union of discs with boundary on $\boundary N'_n - P$ and annuli with at least one boundary component on $\boundary N'_n - P$. These boundary components must cross the sutures on $\boundary N'_n - P$. Analyzing these intersections gives a lower bound on $I(Q_n)$ which is, therefore, a lower bound on $I(Q)$. This lower bound implies the inequality in Lackenby's theorem. As in Gabai's theorem, Lackenby's theorem can be rewritten as a theorem about a sutured manifold $(M,\gamma,\beta)$ with $\beta$ a knot in $M$. (The knot $\beta$ is the core of the surgery solid torus with slope $b$.) 

The point of this paper is to develop a theory whereby we can replace the knot $\beta$ in the the proof of Gabai's and Lackenby's theorems with an arc $\beta$. In Theorem \ref{Thm: Main Theorem B}, this arc is the cocore of a 2--handle added to $b \subset \boundary N$.

The proof of Theorem \ref{Thm: Main Theorem C} is inspired by the proof of Lackenby's theorem. For the time being, let $(M, \gamma') = (N[b], \gamma - b)$ and consider the arc $\beta \subset M$ which is the cocore of the attached 2--handle. If we could construct a useful hierarchy of $(M,\gamma',\beta)$ disjoint from $\beta$, we could adapt Lackenby's combinatorics to obtain a result similar to Theorem \ref{Thm: Main Theorem C}. However, it seems unlikely that such a hierarchy can exist, since although a sequence
\[
(M,\gamma',\beta) \stackrel{S_1}{\to} (M_1,\gamma_1,\beta_1) \stackrel{S_2}{\to} \hdots \stackrel{S_n}{\to} (M_n,\gamma_n,\beta_n)
\]
can be constructed so that each decomposing surface represents a given homology class, and although it is possible to find such surfaces representing the homology class that are disjoint from $\beta$, it may not be possible to find such surfaces giving a $\beta$-taut decomposition which are (in the terminology of \cite{MSc3}) ``conditioned''. Instead we develop the theory of ``band-taut sutured manifolds'' to give the necessary control over intersections between $\beta$ and the decomposing surfaces. Sections \ref{Band Taut Decomp} and \ref{Band Taut Hierarchies} are almost entirely devoted to proving that if $(M,\gamma',\beta)$ is a band-taut sutured manifold then there is a so-called ``band-taut'' sutured manifold hierarchy of $M$. Section \ref{Combinatorics} studies the combinatorics of parameterizing surfaces at the end of a band-taut hierarchy and proves a version of Theorem \ref{Thm: Main Theorem C} for band-taut sutured manifolds. Section \ref{Branched Surface} reviews Gabai's construction of the branched surface associated to a sequence of sutured manifold decompositions and sets up the technology to prove that the arc $\beta$ can sometimes be isotoped into the branched surface associated to a taut hierarchy.

In classical combinatorial sutured manifold theory, sutured manifold decompositions are usually constructed so that they ``respect'' a given parameterizing surface. The framework of ``band taut sutured manifolds'' requires that we have sutured manifold decompositions that respect each of two, not necessarily disjoint, parameterizing surfaces. Section \ref{Param} is devoted to explaining this mild generalization of the classical theory.

Sections \ref{arc taut to band taut} and \ref{Punctured Tori} convert the main theorem for the theory of band taut sutured manifolds into theorems for arc-taut and nil-taut sutured manifolds. Section \ref{Tunnels}  gives the application to tunnel number one knots and links.

\subsection*{Acknowledgements} This paper has its roots in my Ph.D. dissertation \cite{T1}, although none of the present work appears there. I am grateful to Qilong Guo who found a gap in \cite{T1}, which lead to the development of the concept of ``band taut sutured manifold''. I am grateful to Marty Scharlemann for his encouragement and helpful comments. Thanks also to the referees for their careful reading and suggestions.
  
\section{Sutured Manifolds}\label{Sec: SM}
A sutured manifold $(M,\gamma,\beta)$ consists of a compact orientable 3-manifold $M$, a collection of annuli $A(\gamma) \subset \boundary M$ whose cores are oriented simple closed curves $\gamma$, a collection of torus components $T(\gamma) \subset \boundary M$, and a 1-complex $\beta$ properly embedded in $M$. Furthermore, the closure of $\boundary M - (A(\gamma) \cup T(\gamma))$ is the disjoint union of two surfaces $R_- = R_-(\gamma)$ and $R_+ = R_+(\gamma)$. Each component of $A(\gamma)$ is adjacent to both $R_-$ and $R_+$. The surfaces $R_-$ and $R_+$ are oriented so that if $A$ is a component of $A(\gamma)$, then the curves $R_- \cap A$, $R_+ \cap A$ and $\gamma \cap A$ are all non-empty and are mutually parallel as oriented curves. We denote the union of components of $A(\gamma) - \gamma$ adjacent to $R_\pm$ by $A_\pm$. We let $R(\gamma) = R_- \cup R_+$. We use $R_\pm$ to denote $R_-$ or $R_+$. 

The orientation on $\boundary R_+$ gives an outward normal orientation to $R_+$ and the orientation on $\boundary R_-$ gives an inward normal orientation to $R_-$. We assign each edge of $\beta$ an orientation with the stipulation that if an edge has an endpoint in $R_- \cup A_-$ then it is the initial endpoint of the edge and if an edge has an endpoint in $R_+ \cup A_+$ then it is the terminal endpoint of the edge. We will only be considering 1-complexes $\beta$ where this stipulation on the orientation of edges can be attained. (That is, no edge of $\beta$ will have both endpoints in $R_\pm \cup A_\pm$.)

If $(M,\gamma,\beta)$ is a sutured manifold and if $S \subset M$ is a connected surface in general position with respect to $\beta$, the \defn{$\beta$-norm} of $S$ is \[x_\beta(S) = \max\{0,-\chi(S) + |S \cap \beta|\}.\] If $S$ is a disconnected surface in general position with respect to $\beta$, the $\beta$-norm is defined to be the sum of the $\beta$-norms of its components. The norm $x_\nil$ is called the \defn{Thurston norm}.

The surface $S$ is \defn{$\beta$-minimizing} if, out of all embedded surfaces with the same boundary as $S$ and representing $[S,\boundary S]$ in $H_2(M,\boundary S)$, the surface $S$ has minimal $\beta$-norm. $S$ is \defn{$\beta$-taut} if it is $\beta$-incompressible (i.e. $S - \beta$ is incompressible in $M - \beta$), $\beta$-minimizing, and any given edge of $\beta$ always intersects $S$ with the same sign.

A sutured manifold is \defn{$\beta$-taut} if:
\begin{enumerate}
\item[(T0)] $\beta$ is disjoint from $A(\gamma) \cup T(\gamma)$.
\item[(T1)] $M$ is $\beta$-irreducible
\item[(T2)] $R(\gamma)$ (equivalently $R_-$ and $R_+$) and $T(\gamma)$ are $\beta$-taut.
\end{enumerate}

If a 3--manifold or surface is $\nil$-taut, we say it is \defn{taut in the Thurston norm} or sometimes, simply, \defn{taut}.

The sutured manifold terminology up until now has been standard (see \cite{MSc3}). Here is the central new idea of this paper. We say that a sutured manifold $(M,\gamma,\beta)$ is \defn{banded} if 
\begin{enumerate}
\item[(B1)] there exists at most one edge $c_\beta \subset \beta$, having an endpoint in $A(\gamma)$. If $c_\beta \neq \nil$, one endpoint lies in $A_-$ and the other lies in $A_+$. The edge $c_\beta$ is called \defn{the core}.
\item[(B2)] If $c_\beta \neq \nil$, then there exits a disc $D_\beta$, which we think of as an octagon, having its boundary divided into eight arcs, $c_1, c_2, \hdots, c_8$ (in cyclic order), called the \defn{edges} of $D_\beta$. The arc $c_\beta$ is contained in $D_\beta$ and the interior of $D_\beta$ is otherwise disjoint from $\beta$. We require that:
\begin{itemize}
\item $c_1$ and $c_5$ are properly embedded in $R_- - \boundary \beta$,
\item $c_2$ and $c_6$ each are properly embedded in $A(\gamma)$, intersect $\gamma$ exactly once each, and each contains an endpoint of $c_\beta$,
\item $c_3$ and $c_7$ are properly embedded in $R_+ - \boundary \beta$,
\item $c_4$ and $c_8$ each either traverse an edge of $\beta - c_\beta$ or are properly embedded in $A(\gamma)$ and intersect $\gamma$ exactly once.
\end{itemize}

Define $e_\beta$ to be the union of edges of $\beta - c_\beta$ that are traversed by $\boundary D_\beta \cap (\beta - c_\beta)$. We have that $|e_\beta| \leq 2$. The disc $D_\beta$ is called \defn{the band} and the components of $e_\beta$ are called the \defn{sides} of the band. The sides of a band may lie on zero, one, or two edges of $\beta$. The arc $c_1 \cup c_2 \cup c_3$ is called the \defn{top} of the band and the arc $c_5 \cup c_6 \cup c_7$ is called the \defn{bottom} of the band. See Figure \ref{Fig: The band}. 
\end{enumerate}

\begin{figure}[ht!] 
\labellist \small\hair 2pt 
\pinlabel {$D_\beta$} at 108 180
\pinlabel {$\gamma$} [Bl] at 318 280
\pinlabel {$\gamma$} [b] at 167 337
\pinlabel {$\gamma$} [tr] at 65 5
\pinlabel {$c_2$} [br] at 131 291
\pinlabel {$c_3$} [Bl] at 211 258
\pinlabel {$c_\beta$} [l] at 155 180
\pinlabel {$R_+$} [Bl] at 253 322
\pinlabel {$R_+$} [tr] at 8 56
\pinlabel {$R_-$} [Br] at 80 298
\pinlabel {$R_-$} [tl] at 201 85
\pinlabel {$c_7$} [Bl] at 59 107
\endlabellist 
\centering 
\includegraphics[scale=.5]{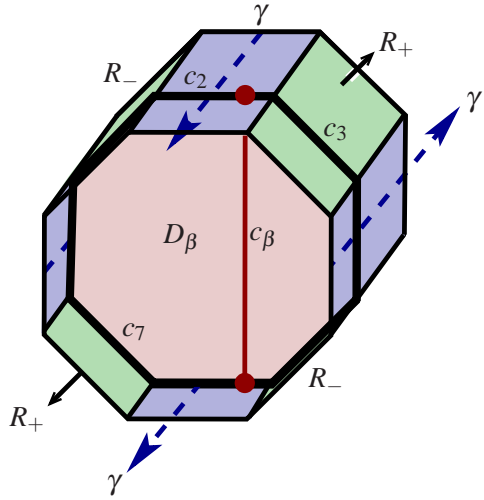}
\caption{The red disc is a band with $e_\beta = \nil$. The green surfaces are subsurfaces of $R(\gamma)$ and the blue surfaces are subsurfaces of $A(\gamma)$. The edges of the band are labelled clockwise $c_1$ through $c_8$ with $c_2$ containing the top endpoint of the red arc $c_\beta$.}
\label{Fig: The band}
\end{figure}

A banded sutured manifold $(M,\gamma,\beta)$ is \defn{band-taut} if $(M,\gamma,\beta - c_\beta)$ is $(\beta - c_\beta)$-taut. 

\begin{remark}
The core of the band $c_\beta$ is the arc we try to isotope onto the branched surface coming from a sutured manifold hierarchy. When building the hierarchy we will attempt to make each decomposing surface disjoint from $c_\beta$. The disc $D_\beta$ helps to guide the isotopy of (parts of) $c_\beta$ into the branched surface coming from a sutured manifold hierarchy. That the endpoints of $c_\beta$ lie in $A(\gamma)$ allow us to use the surface $R(\gamma)$ to modify decomposing surfaces so as to give them algebraic intersection number zero with $c_\beta$. Because we want to appeal to as much of the sutured manifold theory developed in \cite{MSc3} and \cite{MSc4} as possible, we need ways of appealing to results about taut sutured manifolds. The sides of the band allow us to make use of these results. 
\end{remark}

\section{Decompositions}\label{Band Taut Decomp}
In classical combinatorial sutured manifold theory, sutured manifolds are decomposed using so-called ``conditioned'' surfaces and a variety of ``product surfaces''. We review and expand the classical definitions and then discuss how the surfaces can give decompositions of band-taut sutured manifolds. 

\subsection{Sutured manifold decompositions}

\subsubsection{Decomposing surfaces} If $(M,\gamma,\beta)$ is a sutured manifold, a \defn{decomposing surface} (cf. \cite[Definition 2.3]{MSc3}) is a properly embedded surface $S \subset M$ transverse to $\beta$ such that:
\begin{enumerate}
\item[(D1)] $\boundary S$ intersects each component of $T(\gamma)$ in a (possibly empty) collection of coherently oriented circles.

\item[(D2)] $\boundary S$ intersects each component of $A(\gamma)$ in circles parallel to $\gamma$ (and oriented in the same direction as $\gamma$), in essential arcs, or not at all.

\item[(D3)] Each circle component of $\boundary S \cap A(\gamma)$ is disjoint from $\gamma$ and no arc component of $\boundary S \cap A(\gamma)$ intersects $\gamma$ more than once.
\end{enumerate}

If $S$ is a decomposing surface, there is a standard way of placing a sutured manifold structure on $M' = M - \inter{\eta}(S)$. The curves $\gamma'$ are the oriented double curve sum of $\gamma$ with $\boundary S$. Let $\beta' = \beta \cap M'$. The sutured manifold $(M',\gamma',\beta')$ is obtained by \defn{decomposing} $(M,\gamma,\beta)$ using $S$. We write $(M,\gamma,\beta) \stackrel{S}{\to} (M',\gamma',\beta')$. If $(M,\gamma,\beta)$ is $\beta$-taut and if $(M',\gamma',\beta')$ is $\beta'$-taut, then we say the decomposition is \defn{$\beta$-taut}.

If the annuli and tori $A(\gamma) \cup T(\gamma)$ are not disjoint from $\beta$, we need to be more precise about the formation of the annuli $A(\gamma')$ in $(M',\gamma',\beta')$. We form annuli $A(\gamma') = \eta(\gamma')$ by demanding that $(A(\gamma) \cup T(\gamma)) \cap M'$ is a subset of $A(\gamma') \cup T(\gamma')$. This requirement ensures that any endpoint of $\beta$ that lies in $A(\gamma) \cup T(\gamma)$ continues to lie in $A(\gamma') \cup T(\gamma')$. See Figure \ref{Fig: Decomposing Sutures}.

\begin{figure}[ht!] 
\labellist \small\hair 2pt 
\pinlabel {$\boundary S$} [t] at 130 192
\pinlabel {$\gamma$} [r] at 4 370
\pinlabel {$\boundary A(\gamma)$} [r] at 4 443
\pinlabel {$\boundary A(\gamma)$} [r] at 4 300
\pinlabel {$\gamma'$} [r] at 389 373
\pinlabel {$\boundary A(\gamma')$} [r] at 389 446
\pinlabel {$\boundary A(\gamma')$} [r] at 389 302
\pinlabel {$\boundary \beta$} [l] at 86 407
\pinlabel {$\boundary \beta$} [l] at 195 335
\endlabellist 
\centering 
\includegraphics[scale=.4]{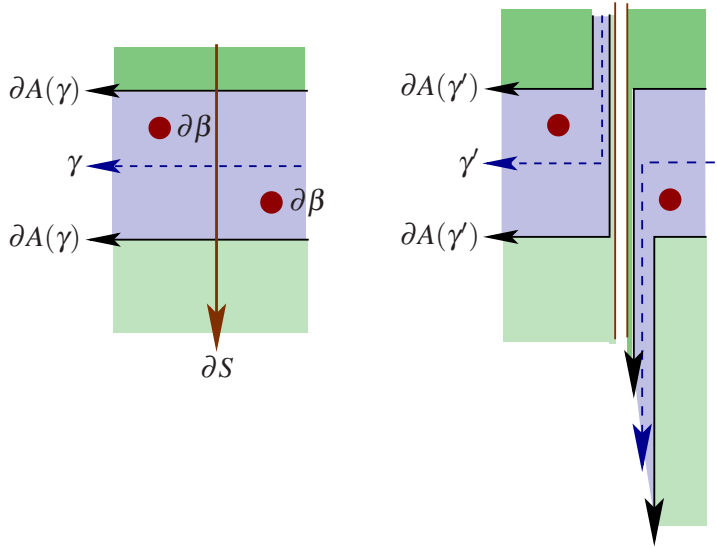}
\caption{Creating sutures in $M - \inter{\eta}(S)$.}
\label{Fig: Decomposing Sutures}
\end{figure}

\subsubsection{Product Surfaces} If $e \subset \beta$ is an edge with both endpoints in $R(\gamma)$, a \defn{cancelling} disc for $e$ is a disc properly embedded in $M - \inter{\eta}(\beta)$ having boundary running once across $e$ and once across $A(\gamma)$. See Figure \ref{Fig: Cancelling Disc}. A \defn{product disc} in a sutured manifold $(M,\gamma,\beta)$ is a rectangle $P$ properly embedded in $M$ such that $P \cap \beta = \nil$ and $\boundary P \cap A(\gamma)$ consists of two opposite edges of the rectangle each intersecting $\gamma$ once transversally. Notice that the frontier of a regular neighborhood of a cancelling disc is a product disc. A product disc $P$ is \defn{allowable} if no component of $\boundary P \cap R(\gamma)$ is $\beta$-inessential.

\begin{figure}[ht!] 
\labellist \small\hair 2pt 
\pinlabel {$e$} [b] at 286 75
\pinlabel {$D$} at 286 52
\endlabellist 
\centering
\includegraphics[scale=0.5]{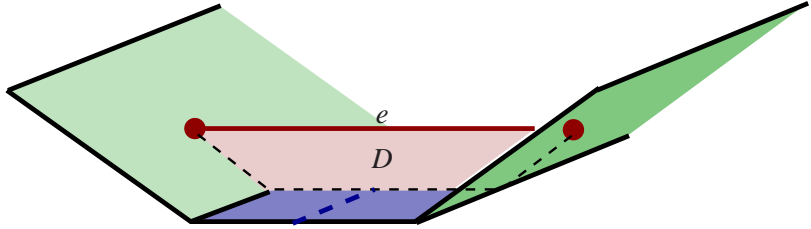}
\caption{A cancelling disc $D$ for an edge $e$}
\label{Fig: Cancelling Disc}
\end{figure}

An \defn{amalgamating disc} $D$ in $(M,\gamma,\beta)$ is a rectangle with two opposite edges lying on components of $\beta$ that are edges joining $R_-$ to $R_+$, one edge in $R_+$ and one edge in $R_-$. If $\boundary D$ traverses a single edge of $\beta$ twice, it is a \defn{self amalgamating disc}, otherwise it is a \defn{nonself amalgamating disc}. A self amalgamating disc is \defn{allowable} if both of the arcs $\boundary D \cap \boundary M$ are $\beta$-essential in $R(\gamma)$.

 If in $(M,\gamma,\beta)$ there is a cancelling disc $D$ for $e$, we say that the sutured manifold $(M,\gamma,\beta - e)$ is obtained from $(M,\gamma,\beta)$ by \defn{cancelling} the arc $e$. If in $(M,\gamma,\beta)$ there is a nonself amalgamating disc with boundary on components $\beta_1$ and $\beta_2$ of $\beta$, we say that the sutured manifolds $(M,\gamma,\beta - \beta_1)$ and $(M,\gamma,\beta - \beta_2)$ are obtained by \defn{amalgamating the arcs} $\beta_1$ and $\beta_2$.

Lemma 4.3 of \cite{MSc3} shows that if $(M,\gamma,\beta)$ is $\beta$-taut, then after cancelling arc $e$, the sutured manifold is still $(\beta - e)$-taut. The converse is also easily proven. By \cite[Lemma 4.2]{MSc3}, if $(M,\gamma,\beta)$ is taut, then so is the sutured manifold obtained by decomposing along a product disc in $(M,\gamma,\beta)$.  By \cite[Lemma 4.3 and Lemma 4.4]{MSc3}, if $(M,\gamma,\beta - \beta_1)$ is obtained by amalgamating arcs $\beta_1$ and $\beta_2$ in the $\beta$-taut sutured manifold $(M,\gamma,\beta)$, then $(M,\gamma,\beta - \beta_1)$ is $(\beta - \beta_1)$-taut. Later we will review a method for eliminating self amalgamating discs so that tautness is preserved.

 A \defn{product annulus} $P$ is an annulus properly embedded in $M$ that is disjoint from $\beta$ and that has one boundary component in $R_-$ and the other in $R_+$. (See \cite[Definition 4.1]{MSc3}.)  A product annulus is \defn{allowable} if $P$ is not the frontier of a regular neighborhood of an arc in $M$ (this is the same as being ``non-trivial'' in the sense of \cite[Definition 4.1]{MSc3}).  Notice that attaching the two edges of a self-amalgamating disc lying on $\beta$ produces a product annulus.

\subsubsection{Conditioned and rinsed surfaces} 

In addition to decomposing sutured manifolds along product surfaces, we will also need to decompose along more complicated surfaces. We require such surfaces to be ``conditioned''  \cite[Definition 2.4]{MSc3}. A \defn{conditioned} 1--manifold $C \subset \boundary M$ is an embedded oriented 1--manifold satisfying:
\begin{itemize}
\item[(C0)] All circle components of $C$ lying in the same component of $A(\gamma) \cup T(\gamma)$ are oriented in the same direction, and if they lie in $A(\gamma)$, they are oriented in the same direction as the adjacent component of $\gamma$.
\item[(C1)] All arcs of $C \cap A(\gamma)$ in any annulus component of $A(\gamma)$ are oriented in the same direction.
\item[(C2)] No collection of simple closed curves of $C \cap R(\gamma)$ is trivial in \linebreak[4] $H_1(R(\gamma),\boundary R(\gamma))$.
\end{itemize}
Notice that if $z \in H_1(\boundary M)$ is non-trivial, then there is a conditioned 1--manifold in $M$ representing $z$. Furthermore, if $C$ is a conditioned 1--manifold then the oriented double curve sum of $C$ with $\boundary R(\gamma)$ is also conditioned.

A decomposing surface $S \subset M$ is \defn{conditioned} if $\boundary S$ is conditioned and if, additionally,
\begin{itemize}
\item[(C3)] Each edge of $\beta$ intersects $S \cup R(\gamma)$ always with the same sign.
\end{itemize}

A surface $S$ in a banded 3-manifold $(M,\gamma,\beta)$ is \defn{rinsed} if $S$ is conditioned in $(M,\gamma,\beta - c_\beta)$, if $S$ has zero algebraic intersection with $c_\beta$, and if every separating closed component of $S$ bounds with a closed component of $R(\gamma)$ a product region intersecting $\beta$ in vertical arcs. 

\subsection{Band-taut decompositions}
An arbitrary decomposition of a banded sutured manifold may not create a banded sutured manifold. In this section, we show how certain surfaces can be used to usefully decompose band-taut sutured manifolds.

The easiest instance of such a decomposition is if $(M,\gamma,\beta)$ is a banded sutured manifold and if $E$ is a cancelling disc with interior disjoint from $D_\beta$ for a component $\beta_1$ of $e_\beta$. Let $E'$ be the product disc in $M$ that is the frontier of a regular neighborhood of $E$. The disc $E'$ intersects $D_\beta$ in either one or two arcs. Those arcs join the top of $D_\beta$ the bottom of $D_\beta$. If there are two arcs (which happens if $c_4$ and $c_8$ run along the same edge of $\beta$), one arc joins $c_1$ to $c_7$ and the other joins $c_3$ to $c_5$. If there is a single arc, it either joins $c_1$ to $c_7$ or joins $c_3$ to $c_5$. Let $(M',\gamma',\beta')$ be the result of decomposing $(M,\gamma,\beta)$ using $E'$. The disc $D_\beta$ is decomposed into 2 or 3 discs, one of which $D'_\beta$ contains $c_\beta = c_{\beta'}$. The disc $D'_\beta$ is clearly a band and $|e_{\beta}'| < |e_{\beta}|$. In effect, we have cancelled an edge of $e_\beta$ and the new band runs along a suture instead of along the edge. We call $E'$ a \defn{band-decomposing product disc}. Since cancelling an edge and decomposing along a product disc disjoint from $D_\beta$ preserve tautness, decomposing a band-taut sutured manifold along either a band-decomposing disc or a product disc disjoint from $D_\beta$ preserves band-tautness.

We also need ways of decomposing along other product surfaces or conditioned surfaces in ways that preserve band-tautness. To that end, suppose that a decomposing surface $S$ in a banded sutured manifold $(M,\gamma,\beta)$ is transverse to $D_\beta$. We say that $S$ is a \defn{band-decomposing surface} if it is either a band-decomposing product disc, a product disc disjoint from $D_\beta$ or if it satisfies:
\begin{itemize}
\item[(BD)] Either $e_\beta = \nil$ and $c_\beta$ is isotopic in $D_\beta$ relative to its endpoints into $\boundary M$ or all of the following are true:
\begin{enumerate}
\item there exists a properly embedded arc $c$ in $D_\beta$ joining the top of $D_\beta$ to the bottom of $D_\beta$ that is disjoint from $S$
\item each point of the intersection between $\boundary S$ and the top of $D_\beta$ has the same sign as the sign of intersection between $\gamma$ and $c_2$.
\item each point of the intersection between $\boundary S$ and the bottom of $D_\beta$ has the same sign as the sign of intersection between $\gamma$ and $c_6$.
\item each point of the intersection between $\boundary S$ and $c_4$ has the same sign.
\item each point of the intersection between $\boundary S$ and $c_8$ has the same sign.
\end{enumerate}
\end{itemize}

Suppose that $(M,\gamma,\beta) \stackrel{S}{\to} (M',\gamma',\beta')$ is a $(\beta - c_\beta)$-taut sutured manifold decomposition where $(M,\gamma,\beta)$ is a banded sutured manifold and $S$ is a band-decomposing surface. We say that the decomposition is \defn{band-taut} if $(M,\gamma,\beta)$ and $(M',\gamma',\beta')$ are each band-taut and one of the following holds:
\begin{enumerate}
\item[(BT1)] $e_\beta = \nil$, $c_\beta$ is isotopic (relative to its endpoints) in $D_\beta$ into $\boundary M$ and $D_{\beta'} = c_{\beta'} = \nil$, or
\item[(BT2)] $c_{\beta'}$ is a properly embedded arc in $D_\beta - S$, such that the initial endpoint of $c_{\beta'}$ lies in $A_-(\gamma') \cap \boundary D_{\beta}$ and the terminal endpoint of $c_{\beta'}$ lies in $A_+(\gamma') \cap \boundary D_{\beta}$ and $D_{\beta'}$ is the component of $D_\beta \cap M'$ containing $c_{\beta'}$.
\end{enumerate}

\begin{lemma}\label{Lem: Arc Positions}
Suppose that $(M,\gamma,\beta)$ is band-taut with $D_\beta \neq \nil$. If a band-decomposing surface $S$ satisfying (BD) has been isotoped relative to $\boundary S$ so as to minimize $|S \cap D_\beta|$ then every component of $S \cap D_\beta$ is an arc joining either the top or bottom of $D_\beta$ to $c_4$ or $c_8$.
\end{lemma}
\begin{proof}
That no component of $S \cap D_\beta$ is a circle follows from an innermost disc argument. By condition (1) of (BD), no arc component joins $c_4$ to $c_8$. By conditions (4) and (5) of (BD), no arc component joins $c_4$ to itself or joins $c_8$ to itself. Since $|S \cap D_\beta|$ is minimized and since $S$ is a decomposing surface, no arc of $S \cap D_{\beta}$ has both endpoints on the same edge of $D_\beta$.  By conditions (2) and (3), no arc joins the top of $D_\beta$ to itself and no arc joins the bottom of $D_\beta$ to itself. We need only show, therefore, that each arc joins the top or bottom of $D_\beta$ to $c_4$ or $c_8$.

Due to the orientations of $R_-$ and $R_+$, the orientations of $\gamma$ at $\gamma \cap c_2$ and $\gamma \cap c_6$ point in the same direction. Suppose that $\zeta$ is an arc of $S \cap D_\beta$ joining the top of $D_\beta$ to the bottom of $D_\beta$. The orientation of $S$ induces a normal orientation of $\zeta$ in $D_\beta$. The normal orientation of $S$ induces orientations of the endpoints of $\zeta$ that are normal to $D_\beta$ and point in opposite directions. This violates either condition (2) or (3) of (BD). Hence, no arc of $S \cap D_\beta$ joins the top of $D_\beta$ to the bottom of $D_\beta$.  
\end{proof}

\begin{lemma}\label{Lem: Defining band taut decomp}
Suppose that $(M,\gamma,\beta)$ is a band-taut sutured manifold and that $S$ is a band decomposing surface satisfying (BD) and defining a taut decomposition $(M,\gamma,\beta - c_\beta) \stackrel{S}{\to} (M',\gamma',\beta' - c_\beta)$. Then after an isotopy of $S$ (rel $\boundary S$) to minimize $|S \cap D_\beta|$, there are $D_{\beta'} \subset D_\beta$ and $c_{\beta'} \subset D_{\beta'}$ so that the decomposition is band-taut. Furthermore, if (BT1) does not hold, each component of $(D_\beta \cap M') - D_{\beta'}$ is a product disc or cancelling disc.
\end{lemma}

\begin{proof}
If $e_\beta = \nil$ and if $c_{\beta}$ is isotopic into $\boundary M$ in $D_\beta$, define $D_{\beta'} = c_{\beta'} = \nil$. Assume, therefore, that (BT1) does not hold. 

Let $d_T$ and $d_B$ be the top and bottom of $D_\beta$ respectively. Since all points of intersection of $S$ with $d_T$ have the same sign as the intersection of $\gamma$ with $c_2$, each component of  $d_T - S$ contains exactly one point of $\gamma' \cap D_\beta$. Similarly, each component of  $d_B - S$ contains exactly one point of $\gamma' \cap D_\beta$. By Lemma \ref{Lem: Arc Positions}, every component of $S \cap D_\beta$ joins the top or bottom of $D_\beta$ to $c_4$ or $c_8$. This implies both that each component of $e_\beta - S$ has an endpoint in both $R_-(\gamma')$ and $R_+(\gamma')$ and that there is exactly one component $D_{\beta'}$ of $D_\beta \cap M'$ containing both a point of $d_T$ and a point of $d_B$. It is not hard to see that $D_{\beta'}$ is a band containing an arc $c_\beta'$ satisfying the definition of core. See Figures \ref{Fig: New Core1} and \ref{Fig: New Core2}. 

Similarly, each component of $D_\beta - S$ other than $D_{\beta'}$ intersects $c_4 \cup c_8$ in at most one arc. If such a component does intersect $(c_4 \cup c_8)$ it is a product disc or cancelling disc (depending on whether or not $c_4$ or $c_8$ lies in $e_\beta$). If such a component does not intersect $(c_4 \cup c_8)$ then it is adjacent to exactly two arcs of $S \cap D_\beta$ and so is a product disc in $M'$.
\end{proof}

\begin{center}
\begin{figure}[ht]
\scalebox{0.5}{\includegraphics{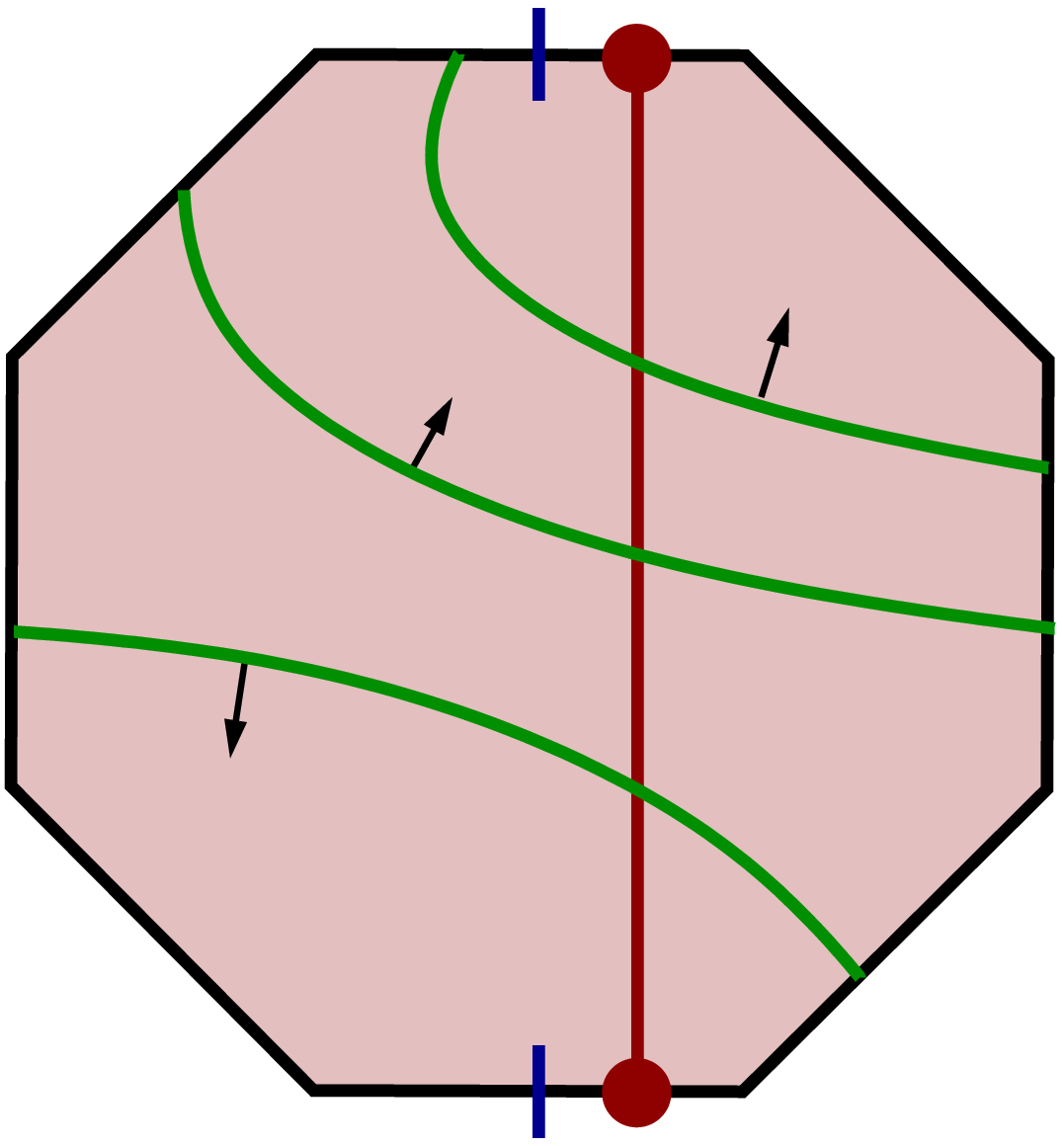}}
\scalebox{0.5}{\includegraphics{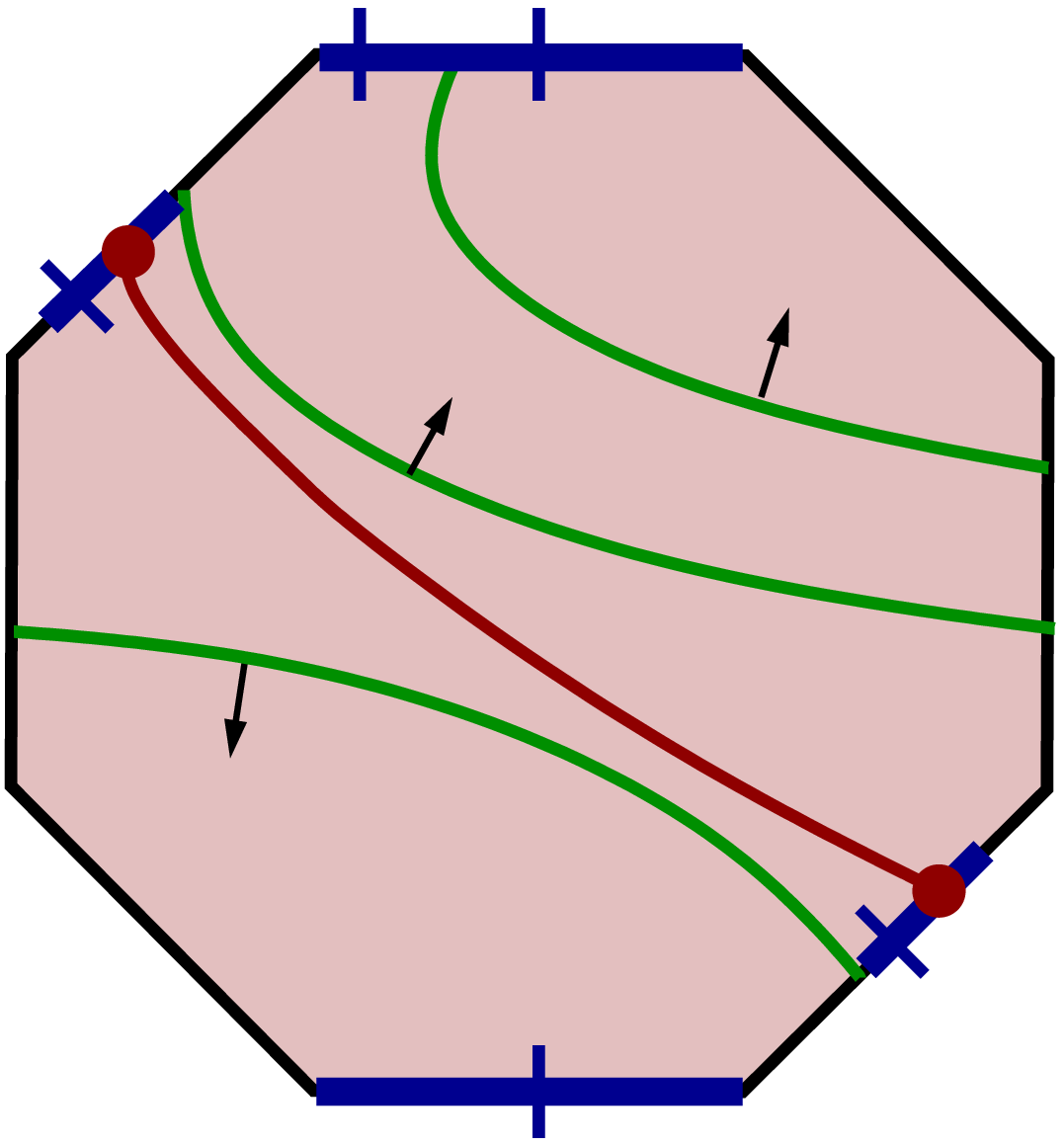}}
\caption{The left image shows the core $c_{\beta}$ and the right image shows the core $c_{\beta'}$. The core $c_{\beta}$ can be isotoped to an arc $c_{\beta'}$ disjoint from $S$ and with endpoints in $A(\gamma')$. The sutures $\gamma'$ are marked on $\boundary D_\beta'$ in the rightmost figure. The endpoint of $c_{\beta}$ lying in $A_\pm$ is isotoped to an endpoint of $c_{\beta'}$ lying in $A_\pm (\gamma')$. In the rightmost figure, the intersection of the annuli $A_\pm(\gamma')$ with $\boundary D_\beta$ are highlighted in blue.}
\label{Fig: New Core1}
\end{figure}
\end{center}

\begin{center}
\begin{figure}[ht]
\scalebox{0.5}{\includegraphics{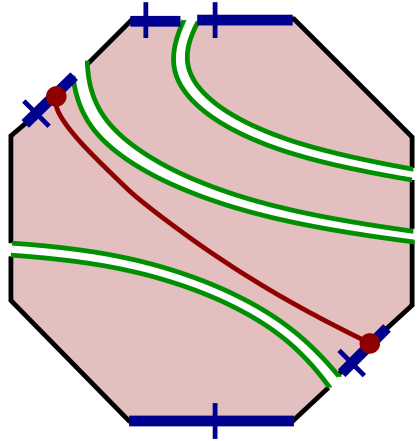}}
\caption{The band $D_{\beta'}$ is the component of $D_{\beta} \cap M'$ containing $c_{\beta'}$.}
\label{Fig: New Core2}
\end{figure}
\end{center}

We note the following:
\begin{lemma}\label{Lem: Isotoping Core}
Suppose that $(M,\gamma,\beta) \stackrel{S}{\to} (M',\gamma',\beta')$ is a band-taut decomposition. Then there is an isotopy of $c_\beta$ relative to its endpoints to an arc $c$ such that the closure of $c \cap \inter{M'}$ is $c_{\beta'}$. Furthermore, $c_{\beta'}$ joins the components of $c_{\beta} \cap \boundary D_{\beta'}$.
\end{lemma}
\begin{proof}
By the definition of band-taut decomposition, $S$ is a band-decomposing surface. If the decomposition satisfies (BT1), then by definition, there is an isotopy of $c_{\beta}$ into $\boundary M$ and we have our conclusion. Suppose, therefore, that the decomposition satisfies (BT2).

If the decomposition is by a band decomposing disc or a product disc disjoint from the band then no isotopy of $c_{\beta}$ is necessary as $c_{\beta'} = c_\beta$. If the decomposition is by a surface satisfying (BD), this follows immediately from the observation that $c_{\beta}$ is isotopic to $c_{\beta'}$ by a proper isotopy in $D_\beta$ that does not move the endpoints of $c_{\beta}$ along edges $c_4$ or $c_8$ of $D_{\beta}$. See Figure \ref{Fig: Isotoping Core}. \end{proof}

\begin{center}
\begin{figure}[ht]
\labellist 
\pinlabel{$c_\beta$} [l] at 188 275
\pinlabel{$c_{\beta'}$} [tr] at 70 209
\endlabellist 
\centering 
\scalebox{0.5}{\includegraphics{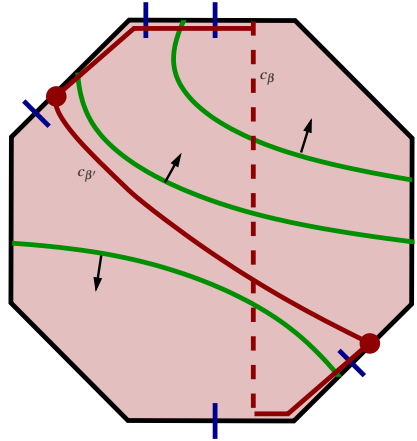}}
\caption{The isotopy of $c_{\beta}$ to $c_{\beta'}$ can be slightly modified to an isotopy of $c_\beta$ relative to its endpoints such that after the isotopy $c_\beta - c_{\beta'}$ lies in $\boundary D_{\beta} \cap \boundary M$. The solid arc is the union of two subarcs in $\boundary D_\beta$ and the arc $c_{\beta'}$. }
\label{Fig: Isotoping Core}
\end{figure}
\end{center}

\subsection{Cancelling discs, amalgamating discs, product discs and product annuli}
The previous section provided a some criteria for creating decompositions of banded sutured manifolds using surfaces that satisfy (BD). Product surfaces, however, may not satisfy (BD). This section shows how to create a band-taut decomposition if there is a cancelling disc, amalgamating disc, product disc, or product annulus in a band-taut sutured manifold.

\subsubsection{Finding disjoint product surfaces} We begin by showing that if there is a cancelling disc, amalgamating disc, product disc, or product annulus in a band-taut sutured manifold, then there is one disjoint from the band.

\begin{lemma}\label{Lem: Cancelling Discs Disjoint}
Suppose that $(M,\gamma,\beta)$ is band-taut. If $\beta - c_\beta$ has a 
\begin{enumerate}
\item cancelling disc or allowable product disc;
\item nonself amalgamating disc; or
\item allowable self amalgamating disc 
\end{enumerate}
then one of the following occurs:
\begin{itemize}
\item there is, respectively, a
\begin{enumerate}
\item cancelling disc or allowable product disc;
\item cancelling disc, allowable product disc, or nonself amalgamating disc; or
\item cancelling disc, allowable product disc, nonself amalgamating disc, or allowable self-amalgamating disc
\end{enumerate}
that is disjoint from $D_\beta$, or
\item $e_\beta = \nil$, the boundary of $D_\beta$ is a $(\beta - c_\beta)$-inessential circle in $\boundary M$, and $c_\beta$ is isotopic in $D_\beta$ into $\boundary M$ (rel $\boundary c_\beta$).
\end{itemize}
\end{lemma}

\begin{proof}
The proofs with each of the three hypotheses are nearly identical, so we prove it only under hypothesis (3). Let $E$ be a cancelling disc, allowable product disc, nonself amalgamating disc, or allowable self amalgamating disc chosen so that out of all such discs, $D_\beta$ and $E$ intersect minimally. By an isotopy of $E$, we can assume that all intersections between $\boundary D_\beta$ and $\boundary E$ occur in $R(\gamma)$. An innermost circle argument shows that there are no circles of intersection between $D_\beta$ and $E$. Similarly, we may assume that if a component of $D_\beta \cap E$ intersects $c_\beta$ then it is an arc in $D_\beta$ joining distinct edges of $D_\beta$.

\textbf{Claim 1:} No component of $D_\beta \cap E$ joins an edge of $D_\beta$ to itself.

Suppose that there is such a component, and let $\xi$ be an outermost such arc in $D_\beta$ with $\Delta$ the disc it cobounds with a subarc of $D_\beta \cap R(\gamma)$. Boundary compress $E$ using $\Delta$ to obtain two discs $E_1$ and $E_2$. Since $E$ was a cancelling disc, product disc, or amalgamating disc, one of $E_1$ or $E_2$ is a cancelling disc, product disc, or amalgamating disc and the other one is a disc with boundary completely contained in $R(\gamma)$. Suppose that $E_2$ is this latter disc. Since $E_2$ is disjoint from $(\beta - c_\beta)$ and since $R(\gamma)$ is $(\beta-c_\beta)$-incompressible, the boundary of $E_2$ is $\beta$-inessential in $R(\gamma)$. Thus, $E$ can be isotoped in the complement of $\beta$ to $E_1$. This isotopy reduces $|D_\beta \cap E|$ and so we have contradicted our choice of $E$.

\textbf{Claim 2:} No arc component of $D_\beta \cap E$ joins edge $c_1$ to edge $c_3$ or edge $c_5$ to edge $c_7$.

Suppose that there is such an arc. Without loss of generality, we may assume that the arc joins side $c_1$ to $c_3$. Let $\xi$ be an arc that, out of all such arcs, is closest to $c_2$. Let $\Delta$ be the disc in $D_\beta$ that it cobounds with $c_2$. Boundary-compress $E$ using $\Delta$ to obtain the union $E'$ of two discs. In $E$, the arc $\xi$ joins $\boundary E \cap R_-$ to $\boundary E \cap R_+$. If $E$ is a cancelling disc, one component of $E'$ is a cancelling disc and the other is a product disc. If $E$ is a product disc, both components of $E'$ are product discs. If $E$ is an amalgamating disc, both components of $E'$ are cancelling discs. Each component of $E'$ intersects $D_\beta$ fewer times than does $E$, so we need only show that if $E$ is an allowable product disc, then at least one component of $E'$ is an allowable product disc. 

Assume that $E$ is an allowable product disc. This implies that it is a $(\beta - c_\beta)$-boundary compressing disc for $M$. Since $E'$ is obtained by boundary compressing $E$, at least one component $E_1$ of $E'$ is a $(\beta - c_\beta)$-boundary compressing disc for $\boundary M$. If it were not allowable, it could be isotoped in the complement of $(\beta - c_\beta)$ to have boundary lying entirely in $R(\gamma)$, this would contradict the fact that $R(\gamma)$ is $(\beta - c_\beta)$-incompressible. Thus, $E_1$ is an allowable product disc.

\textbf{Claim 3:} No arc component of $D_\beta \cap E$ joins edge $c_1$ to edge $c_7$ or edge $c_3$ to edge $c_5$.

Suppose that there is such an arc. Without loss of generality, we may assume that it joins edges $c_1$ and $c_7$. Out of all such arcs, choose one $\xi$ that is as close as possible to edge $c_8$. Boundary compress $E$ using the subdisc of $D_\beta$ cobounded by $c_8$ and $\xi$ to obtain $E'$. 

If $E$ is a cancelling disc, let $E_2$ be the component of $E'$ containing $c_8$ and let $E_1$ be the other component. If $c_8$ is an edge of $\beta$, then $E_2$ is an amalgamating disc and $E_1$ is a cancelling disc. In this case, note that $E_1$ (after a small isotopy to be transverse to $D_\beta$) intersects $D_\beta$ fewer times than does $E$. This contradicts our choice of $E$. If $c_8$ is not an edge of $\beta$, then $E_2$ is a cancelling disc and $E_1$ is a product disc.  If $E_1$ is not allowable, it can be isotoped in the complement of $(\beta - c_\beta)$ to have boundary contained entirely in $R(\gamma)$. Since $R(\gamma)$ is $(\beta - c_\beta)$-incompressible, the boundary of this disc must be $(\beta - c_\beta)$-inessential in $R(\gamma)$. It is then easy to see that there is an isotopy reducing the number of intersections between $E$ and $D_\beta$, a contradiction. 

If $E$ is an allowable product disc, then each component of $E'$ is either a cancelling disc or a product disc. As before, if a component of $E'$ is a product disc, it must be allowable.

If $E$ is an amalgamating disc, then each component of $E'$ is either an amalgamating disc or a cancelling disc. Let $E_1$ and $E_2$ be the components of $E'$. Suppose that each of $E_1$ or $E_2$ is a self amalgamating disc. We must show that at least one of them is allowable. Since $\Delta$ runs along $c_8$, each component of $E'$ runs along $c_8$ at least once. Since each component of $E'$ is a self amalgamating disc, this implies that $E$ is a self-amalgamating disc for $c_8$. By hypothesis, $E$ is allowable. Thus, at least one loop of $E' \cap R_-$ is $(\beta - c_\beta)$-essential and at least one loop of $E' \cap R_+$ is $(\beta - c_\beta)$-essential. If the union of these loops is either $\boundary E_1 \cap R(\gamma)$ or $\boundary E_2 \cap R(\gamma)$, then either $E_1$ or $E_2$ is allowable. Thus, we may assume that one arc of $E_i \cap R(\gamma)$ is $(\beta - c_\beta)$-essential and the other one is inessential for both $i = 1$ and $i = 2$. Gluing the two arcs of $\boundary E_i \cap c_8$ together we obtain a product annulus with one end $(\beta - c_\beta)$-essential and the other end $(\beta - c_\beta)$-essential. Capping the inessential end off, creates a disc disjoint from $(\beta - c_\beta)$ with boundary a $(\beta - c_\beta)$-essential loop in $R(\gamma)$. That is, the disc is a $(\beta - c_\beta)$ compressing disc for $R(\gamma)$, a contradiction. Hence, if both $E_1$ and $E_2$ are self-amalgamating discs at least one of them is allowable. Since each of $E_1$ and $E_2$ intersects $D_\beta$ fewer times than does $E$, we have contradicted our choice of $E$.

\textbf{Claim 4:} No arc component of $D_\beta \cap E$ joins side $c_3$ to side $c_7$ or side $c_1$ to side $c_5$. 

Suppose that there is such an arc. Without loss of generality, we may assume that the arc joins side $c_3$ to side $c_7$. Out of all such arcs choose one $\xi$ that is outermost on $E$. Boundary compress $D_\beta$ using the outermost disc in $E$ bounded by $\xi$. This converts $D_\beta$ into two discs, $D_4$ and $D_8$ containing $c_4$ and $c_8$ respectively. The disc $D_4$ also contains the edge $c_6$ and the disc $D_8$ also contains the edge $c_1$. Since both $c_2$ and $c_6$ are contained in $A(\gamma)$, the discs $D_4$ and $D_8$ are either cancelling discs or product discs. We must show that if they are both product discs, then at least one of them is allowable. 

Assume to the contrary, that both $D_4$ and $D_8$ are product discs that are not allowable. Since product discs that are not allowable can be isotoped in the complement of $(\beta - c_\beta)$ to have boundary lying in $R(\gamma)$, and since $R(\gamma)$ is $(\beta - c_\beta)$-incompressible, both $D_4$ and $D_8$ are discs with $(\beta - c_\beta)$-inessential boundary in $\boundary M$. The disc $D_\beta$ can be reconstructed by banding the discs $D_4$ and $D_8$ together using an arc lying entirely in $R_+$. Since $M$ is $(\beta - c_\beta)$-irreducible and since $c_\beta$ lies in $D_\beta$, the arc $c_\beta$ is isotopic in in $D_\beta$ into $\boundary M$ relative to its endpoints.
\end{proof}

We now study how cancelling discs, amalgamating discs, product discs, and product annuli give band-taut sutured manifold decompositions. 

\subsubsection{Eliminating cancelling discs, product discs, and nonself amalgamating discs}\label{Eliminate Cancelling discs}

\begin{lemma}\label{Lem: product disc band taut}
Suppose that $(M,\gamma,\beta)$ is a connected band taut sutured manifold other than a 3-ball containing a single suture in its boundary and a single arc of $\beta - c_\beta$. If $M$ contains a cancelling disc or allowable product disc, then there is a band-taut decompositions
\[
(M,\gamma,\beta)\stackrel{S}{\to} (M',\gamma',\beta')
\]
The decomposition is a decomposition by a product disc, possibly satisfying (BT1).
\end{lemma}
\begin{proof}
If $e_\beta = \nil$ and if $c_\beta$ is parallel into $\boundary M$ along $D_\beta$, let $D_{\beta'} = c_{\beta'} = \nil$. The decomposition by the given cancelling disc or allowable product disc is then band-taut. Thus, by Lemma \ref{Lem: Cancelling Discs Disjoint}, we may assume that there is a cancelling disc or allowable product disc $P$ that is disjoint from $D_\beta$. If $P$ is a cancelling disc, let $S$ be the frontier of a regular neighborhood of $P$ in $M$. Notice that $S$ is an allowable product disc since $M$ is not a 3-ball with a single suture in its boundary and a single arc in $\beta - c_\beta$. If $P$ was a cancelling disc for a component of $e_\beta$, then $S$ is a band-decomposing product disc. If $P$ is an allowable product disc, let $S = P$. Decomposing a taut sutured manifold using a product disc is a taut decomposition, so it is evident that the decomposition $(M,\gamma,\beta) \stackrel{S}{\to} (M',\gamma',\beta')$ is band-taut.
\end{proof}

Notice that if $(M,\gamma,\beta)$ has a cancelling disc for an edge $e \subset \beta$, then the decomposition given by Lemma \ref{Lem: product disc band taut} cuts off from $M$ a 3--ball having a single suture in its boundary and containing the edge $e$ and the cancelling disc. We then cancel the arc $e$. (The reason for decomposing along $S$ is that at the end of the hierarchy we will want to ignore all arc cancellations. Decomposing along the product disc $S$ before cancelling ensures that the cancellable arc is in its own component of the sutured manifold at the end of the hierarchy.)

As a final remark in this subsection, we note that if $(M,\gamma,\beta)$ is a band-taut sutured manifold and if there is a nonself amalgamating disc $P$ disjoint from $D_\beta$ we can eliminate a component of $P \cap (\beta - e_\beta)$ from $\beta$ and preserve band tautness. If $P$ runs across two components of $e_\beta$, then we view the elimination of one of the components of $e_\beta$ as a melding together of the two sides of $D_\beta$. That is, the band $D_\beta$ is isotoped so that both sides run across the same component of $e_\beta$ and the other component of $e_\beta$ is eliminated. Thus, amalgamating arcs preserves band-tautness.

\subsubsection{Eliminating allowable self amalgamating discs}\label{Section: conversion}
In the construction of a band-taut hierarchy, it will be necessary to eliminate allowable self amalgamating discs as in \cite[Lemma 2.4]{MSc4}. We briefly recall the essentials.

We make use of a trick which allows us to convert between arcs and sutures \cite[Definition 2.2]{MSc4}. If $e \subset \beta$ is an edge with one endpoint in $R_-$ and one in $R_+$, then we \defn{convert} $e$ to a suture $\gamma_e$ by letting $M' = M - \inter{\eta}(e)$ and letting $\gamma_e \subset \boundary M'$ be a meridian of $e$. Lemma 2.2 of \cite{MSc4} shows that $(M,\gamma,\beta)$ is $\beta$-taut if and only if $(M',\gamma \cup \gamma_e, \beta - e)$ is $(\beta - e)$-taut. 

Suppose that $P$ is an allowable self amalgamating disc in $(M,\gamma,\beta)$ disjoint from $D_\beta$. Gluing the components of $\boundary P \cap \beta$ together along the edge of $\beta$ they traverse and isotoping it off $\beta$ creates a product annulus $P_A$. Notice that since $P$ was disjoint from $D_\beta$, the annulus $P_A$ can be created so that it is disjoint from $D_\beta$. Furthermore, if a boundary component of $P_A$ is inessential in $R(\gamma)$, the disc in $R(\gamma)$ it bounds is also disjoint from $D_\beta$.

If both components of $\boundary P_A$ are essential in $R(\gamma)$, we decompose along $P_A$. The parallelism of the edge $P \cap \beta$ into $P_A$ becomes a cancelling disc in the decomposed manifold and we cancel the arc $P \cap \beta$ as in Subsection \ref{Eliminate Cancelling discs}.

If a component of $\boundary P_A$ is inessential in $R(\gamma)$, we choose one such component $\delta$ and let $\Delta$ be the disc in $R(\gamma)$ that it bounds. Let $D$ be the pushoff of $P_A \cup \Delta$ so that it is properly embedded. We decompose along $D$. As described in \cite[Lemma 2.4]{MSc4}, after amalgamating arcs and converting an arc to a suture, the decomposed sutured manifold is equivalent to the sutured manifold obtained by decomposing along $P_A$. By \cite[Lemma 4.2]{MSc3} and \cite[Lemmas 2.3 and 2.4]{MSc4}, if $(M,\gamma,\beta)$ is band taut, so is the decomposed sutured manifold. 

\subsection{Decomposing by rinsed surfaces}

In the previous sections, we have seen how surfaces satisfying (BD) can be used to construct decompositions of banded sutured manifolds and how the presence of a product surface can be used to construct a band-taut decomposition of a band taut sutured manifold. In this section we show that, in the presence of non-trivial second homology, we can find a rinsed surface giving a band-taut decomposition of a band-taut sutured manifold. We begin with some preliminary lemmas that simplify the search for such a surface.

\begin{lemma}\label{Lem: Using S_k}
Suppose that $S$ is a conditioned or rinsed surface in $(M,\gamma,\beta)$. Then the surface $S_k$ obtained by double curve summing $S$ with $k$ copies of $R(\gamma)$ for any $k \geq 0$, is conditioned or rinsed, respectively. Furthermore, if $S$ is rinsed and satisfies conditions (2) and (3) of (BD) in the definition of band decomposing surface, then $S_k$ does also.
\end{lemma}
\begin{proof}
By induction, it suffices to prove the lemma when $k = 1$. We have already observed that $\boundary S_k$ is conditioned. Since $S$ satisfies (C3), it is obvious that $S_k$ does also.

If $S$ is rinsed, then the algebraic intersection number of $S$ with $c_\beta$ is zero. Since $R(\gamma)$ is disjoint from $c_\beta$, the surface $S_k$ also has this property. Suppose that $F$ is a closed component of $S_k$ (with $k =1$). Since $S$ satisfies condition (C2), no component of $S_k$ is a separating closed surface intersecting $S$. Any closed component of $S_k$ must, therefore, be parallel to a component of $R(\gamma)$ and bounds a region of parallelism intersecting $\beta$ only in vertical arcs. Consequently, if $S$ is rinsed, so is $S_k$. Finally, if $S$ satisfies conditions (2) and (3) of (BD), it follows immediately from (C0) and (C1) that $S_k$ also satisfies (2) and (3) of (BD). \end{proof}

\begin{lemma}\label{Lem: Finding c}
Suppose that $(M,\gamma,\beta)$ is band-taut and that $S$ is a rinsed surface satisfying conditions (2) and (3) of (BD) in the definition of band-decomposing surface. Then after an isotopy relative to $\boundary S$ to minimize the pair $(|S \cap D_\beta|,|S \cap c_\beta|)$ with respect to lexicographic order, the surface $S$ is a band decomposing surface satisfying (BD).
\end{lemma}

\begin{proof}
If $c_4$ or $c_8$ lies in $A(\gamma)$, since $\boundary S$ is conditioned, condition (4) or (5) of (BD) is satisfied for that component. If $c_4$ or $c_8$ lies in $e_\beta$, then by condition (C3) in the definition of rinsed, condition (4) or (5) of (BD) is satisfied for that component. Thus, we need only show that $S$ satisfies condition (1) of (BD). 

Each arc of $S \cap D_\beta$ intersects $c_\beta$ at most once, by our initial isotopy of $S$. Suppose that a component $\zeta$ of $S \cap D_\beta$ joins $c_4$ to $c_8$. Since $S$ has algebraic intersection number zero with $c_\beta$, there exists another arc $\zeta'$ intersecting $c_\beta$ but with opposite sign.  By conditions (2) and (3) of (BD), at least one endpoint of $\zeta'$ must lie on $c_4$ or $c_8$. Without loss of generality, assume it to be $c_4$. Since $S$ always intersects $c_4$ with the same sign, $\zeta$ and $\zeta'$ intersect $c_4$ with the same sign. Since the signs of intersection of each of $\zeta$ and $\zeta'$ with $c_\beta$ are the same or opposite of their intersections with $c_4$, and since they intersect $c_4$ with the same sign, $\zeta$ and $\zeta'$ intersect $c_\beta$ with the same sign. This contradicts the choice of $\zeta'$. Hence no arc joins $c_4$ to $c_8$.  Thus, every arc joins either the top or bottom of $D_\beta$ to either $c_4$ or $c_8$. A similar argument shows that if $\zeta$ and $\zeta'$ are arcs each with an endpoint on $c_8$ (or each with an endpoint on $c_4$) and each intersecting $c_\beta$ then they intersect $c_\beta$ with the same sign. It follows that if $S \cap c_\beta$ is non-empty, then precisely one of the following happens:
\begin{enumerate}
\item There are equal numbers of arcs joining $c_5$ to $c_8$ as there are arcs joining $c_1$ to $c_4$ and there are no other arcs.
\item There are equal numbers of arcs joining $c_3$ to $c_8$ as there are arcs joining $c_7$ to $c_4$ and there are no other arcs.
\end{enumerate}
It follows that conclusion (1) of (BD) holds and so $S$ is a band-decomposing surface.
\end{proof}

The previous two lemmas produce a rinsed band decomposing surface from a given rinsed surface. The next lemma produces a rinsed band decomposing surface from a given homology class.

\begin{lemma}\label{Lem: Constructing Rinsed Surface}
Suppose that $(M,\gamma,\beta)$ is a band-taut sutured manifold and that $y \in H_2(M,\boundary M)$ is non-zero. Then there exists a rinsed band-decomposing surface $S$ in $M$ representing $\pm y$.
\end{lemma}
\begin{proof}
Let $C$ be a conditioned 1--manifold in $\boundary M$ representing $\boundary y$. Isotope $C$ so that 
\begin{enumerate}
\item[(a)] Each circle component of $C \cap A(\gamma)$ is contained in a collar of $\boundary R(\gamma)$ that is disjoint from $\boundary c_\beta$
\item[(b)] Each arc component of $C \cap A(\gamma)$ is disjoint from $\boundary D_\beta \cap A(\gamma)$.
\end{enumerate}

Let $\Sigma$ be a surface representing $y$ with $\boundary \Sigma = C$. Discard any separating closed component of $\Sigma$. The surface $\Sigma$ is a decomposing surface. We now proceed to modify it to obtain the surface we want. We begin by arranging for the surface to have non-positive algebraic intersection number with $c_\beta$.

Let $i$ be the algebraic intersection number of $\Sigma$ with $c_\beta$. If $i > 0$, let $\ob{\Sigma}$ be the result of reversing the orientation of $\Sigma$ and let $\ob{C} = \boundary \ob{\Sigma}$. Notice that $\Sigma$ represents $-y$ and that the algebraic intersection between $\ob{\Sigma}$ and $c_\beta$ is non-positive. If $\ob{C}$ is not conditioned, perform cut and paste operations of $\ob{\Sigma}$ with copies of subsurfaces of $R(\gamma)$ to produce a surface $\Sigma'$ having conditioned boundary $C'$ and satisfying (a) and (b). Since $R(\gamma)$ is disjoint from $c_\beta$, the algebraic (and geometric) intersection number of $\Sigma'$ with $c_\beta$ is the same as the algebraic (and geometric) intersection number of $\ob{\Sigma}$ with $c_\beta$. This number is, therefore, negative.

We may, therefore, assume without loss of generality that we have a surface $\Sigma$ such that $C = \boundary \Sigma$ satisfies (a) and (b), and:
\begin{enumerate}
\item[(c)] $\Sigma$ is a conditioned surface representing $\pm y$ 
\item[(d)] The algebraic intersection number $i$ of $\Sigma$ with $c_\beta$ is non-positive.
\end{enumerate}

By Lemma \ref{Lem: Using S_k} and the fact that $R(\gamma)$ is disjoint from $c_\beta$, replacing $\Sigma$ with the double curve sum $\Sigma_k$ of $\Sigma$ with $k$ copies of $R(\gamma)$ does not change (a), (b), (c), or (d). 

We now explain why we may also assume that $\Sigma$ satisfies conditions (2) and (3) of (BD). Suppose that $\boundary \Sigma$ intersects the top of $D_\beta$ in points of opposite intersection number. At least one of those points must lie in $c_1$ or $c_3$. After possibly increasing $k$ and isotoping a circle component of $\Sigma \cap A(\gamma)$ into $R(\gamma)$ (not allowing it to pass through $c_\beta$) we may band together points of $\boundary \Sigma \cap c_1$ or $\boundary \Sigma \cap c_3$ having opposite intersection number to guarantee that all points of $\Sigma \cap (c_1 \cup c_2 \cup c_3)$ have the same intersection number as $\gamma \cap c_2$. Perform any additional necessary cut and paste operations with subsurfaces of $R(\gamma)$ to ensure that $\Sigma$ is conditioned. Since $R(\gamma)$ is disjoint from $c_\beta$, this does not change $i$. Since each intersection point of $\boundary R(\gamma)$ with the top of $D_\beta$ has the same sign as $\gamma \cap c_2$, we still have the property that $\boundary \Sigma$ intersects the top of $D_\beta$ with the same sign as $\gamma \cap c_2$. Similarly, we can guarantee that $\boundary \Sigma$ also always intersects the bottom of $D_\beta$ with the same sign as $\gamma \cap c_6$. This implies that we may assume that $\Sigma$ and $\Sigma_k$ (for $k \geq 1$) satisfies conditions (2) and (3) of (BD) in the definition of band-decomposing surface.  By tubing together points of opposite intersection number, we may also assume that the geometric intersection number of $\Sigma$ and $\Sigma_k$ with each edge of $\beta$ is equal to the algebraic intersection number. From $\Sigma$ discard every closed separating component that does not bound a product region with $R(\gamma)$ intersecting $\beta$ in vertical arcs. Thus, by Lemma \ref{Lem: Finding c}, we may assume that $\Sigma$ and $\Sigma_k$ (for $k \geq 0$), in addition to satisfying (a) - (d), satisfy every requirement for being a rinsed band decomposing surface except that the algebraic intersection number of $\Sigma$ and $\Sigma_k$ with $c_\beta$ may possibly be negative. We now show how to trade (a) for the property that $\Sigma$ has algebraic intersection number 0 with $c_\beta$. We will then have proved our lemma.

Let $\rho_\pm$ be the surface $R_\pm \cup (A_\pm - \inter{\eta}(\gamma))$. Let $S$ be the surface obtained by taking the double curve sum of $\Sigma_k$ with $i$ copies of $\rho_-$. Since $A_-$ has intersection with $c_\beta$ consisting of a single point with sign $+1$, $S$ has zero algebraic intersection number with $c_\beta$. Tube together points with opposite intersection number in the intersection of $S$ with each component of $\beta - c_\beta$. Discard any closed separating component. Isotope $S$ slightly so that all circle components of $S \cap A(\gamma)$ are disjoint from $\gamma$. After discarding any closed separating components of $S$, we have constructed a rinsed band-decomposing surface representing $\pm y$.
\end{proof}
 
Our next two results, which are based on \cite[Theorems 2.5 and 2.6]{MSc3}, are the key to constructing band-taut decompositions. As usual, we let $S_k$ denote the oriented double curve sum of $S$ with $k$ copies of $R(\gamma)$. Recall that $S$ and $S_k$ (for any $k \geq 0$) represent the same class in $H_2(M,\boundary M)$. For reference, we begin by stating \cite[Theorem 2.5]{MSc3}. (As this is an important theorem for us, we note that the surface $R$ in Scharlemann's theorem need not be $R(\gamma)$. Indeed,  the statement of the theorem does not mention sutured manifolds.)

\begin{theorem*}[Theorem 2.5 of \cite{MSc3}]\label{Sch. Thm}
Given:
\begin{enumerate}
\item[(a)] A $\beta$-taut surface $(R,\boundary R)$ in a $\beta$-irreducible 3--manifold $(M,\boundary M)$
\item[(b)] a properly embedded family $C$ of oriented arcs and circles in $\boundary M - \eta(\boundary R)$ which is in the kernel of the map
\[H_1(\boundary M, \eta(\boundary R)) \to H_1(M,\eta(\boundary R))\]
induced by inclusion.
\item[(c)] $y$ in $H_2(M,\boundary M)$ such that $\boundary y = [C],$
\end{enumerate}
then there is a surface $(S,\boundary S)$ in $(M,\boundary M)$ such that
\begin{enumerate}
\item[(i)] $\boundary S - \eta(\boundary R) = C$
\item[(ii)] for some integer $m$, $[S,\boundary S] = y + m[R,\boundary R]$ in $H_2(M,\boundary M)$,
\item[(iii)] for any collection $R'$ of parallel copies of components of $R$ (similarly oriented), the double curve sum of $S$ with $R'$ is $\beta$-taut,
\item[(iv)] any edge of $\beta$ which intersects both $R$ and $S$ intersects them with the same sign.
\end{enumerate}
\end{theorem*}

Notice that conclusion (iv) actually follows immediately from conclusion (iii). 

\begin{theorem}\label{Thm: Rinsed Decomp Exist 1}
Suppose that $(M,\gamma,\beta)$ is a band taut sutured manifold and that $y \in H_2(M,\boundary M)$ is non-zero. Then there exists a rinsed band-decomposing surface $(S,\boundary S) \subset (M,\boundary M)$ satisfying (BD) such that:
\begin{enumerate}
\item[(i)] $S$ represents $\pm y$ in $H_2(M,\boundary M)$.
\item[(ii)] $[S,\boundary S] = \pm y + k[R(\gamma),\boundary R(\gamma)]$ for some $k \geq 0$.
\item[(iii)] for any collection $R'$ of parallel copies of components of $R(\gamma)$, the double curve sum of $S$ with $R'$ is $(\beta - c_\beta)$-taut.
\end{enumerate}
\end{theorem}
\begin{proof}
Let $\Sigma$ be the rinsed band-decomposing surface obtained from Lemma \ref{Lem: Constructing Rinsed Surface}. Let $S$ be the surface given by \cite[Theorem 2.5]{MSc3}. (To apply it we let $\beta - c_\beta$ be the 1--complex in the hypothesis of that theorem and  $R = R(\gamma)$ and $C = \boundary \Sigma - \inter{\eta}(\boundary R)$.) 

Discard any closed separating component of $S$ and isotope $S$ relative to its boundary to minimize $|S \cap D_\beta|$. Our conclusions (i) - (iiii) coincide with conclusions (i) - (iii) of Scharlemann's theorem. Since $\Sigma$ and $S$ are homologous in $H_2(M,\boundary \Sigma \cup \eta(\boundary R(\gamma))$, $S$ has algebraic intersection number 0 with $c_\beta$. The criteria for $S$ to be a rinsed, band decomposing surface follow easily from the construction, Lemma \ref{Lem: Finding c}, and that $\boundary S = \boundary \Sigma$ outside a small neighborhood of $\boundary R(\gamma)$ and inside $\boundary S$ is oriented the same direction as $\boundary R(\gamma)$.
\end{proof}

\begin{corollary}\label{Cor: Rinsed Decomp Exist}
Suppose that $(M,\gamma,\beta)$ is a band-taut sutured manifold with $y \in H_2(M,\boundary M)$ non-zero. Then there exists a rinsed band-decomposing surface $S \subset M$ representing $\pm y$ such that for all non-negative $k \in \Z$, the surface $S_k$ gives a band-taut decomposition $(M,\gamma,\beta) \stackrel{S_k}{\to} (M',\gamma',\beta')$. Furthermore, if (BT1) does not hold, then each component of $(D_\beta \cap M') - D_{\beta'}$ is a product disc or cancelling disc.
\end{corollary}
\begin{proof}
Let $S$ be the surface provided by Theorem \ref{Thm: Rinsed Decomp Exist 1}. Let $(M',\gamma',\beta')$ be the result of decomposing $(M,\gamma,\beta)$ using $S_k$. By Lemmas \ref{Lem: Using S_k} and \ref{Lem: Finding c}, $S_k$ is a rinsed band-decomposing surface satisfying (BD).

Since the double curve sum of $S$ with $(k+1)$ copies of $R(\gamma)$ is $\beta$-taut, $R(\gamma')$ is $(\beta'-c_{\beta})$-taut,  the decomposition
\[
(M,\gamma,\beta) \stackrel{S_k}{\to} (M',\gamma',\beta')
\]
is $(\beta - c_\beta)$-taut. By Lemma \ref{Lem: Defining band taut decomp}, the decomposition is band-taut and if (BT1) does not hold then each component of $(D_\beta \cap M') - D_{\beta'}$ is a product disc or cancelling disc. 
\end{proof}

\section{Parameterizing Surfaces}\label{Param}

Let $(M,\gamma,\beta)$ be a sutured manifold with $\beta$ having endpoints disjoint from $A(\gamma) \cup T(\gamma)$. (That is, $(M,\gamma,\beta)$ satisfies (T0).) A \defn{parameterizing surface} is an orientable surface $Q$ properly embedded in $M - \inter{\eta}(\beta)$ satisfying:
\begin{itemize}
\item[(P1)] $\boundary Q \cap A(\gamma)$ consists of spanning arcs each intersecting $\gamma$ once
\item[(P2)] no component of $Q$ is a sphere or disc disjoint from $\beta \cup  \gamma$. 
\end{itemize}

For a parameterizing surface $Q$, let $\mu(Q)$ denote the number of times that $\boundary Q$ traverses an edge of $\beta$. Define the \defn{index} of $Q$ to be:
\[
I(Q) = \mu(Q) + |\boundary Q \cap \gamma| - 2\chi(Q).
\]

\begin{remark}
In the definition of index given in \cite[Definition 7.4]{MSc3}, there is also a term denoted $\mc{K}$ that is the sum of values of a function defined on the interior vertices of $\beta$. As Scharlemann remarks, that the function can be chosen arbitrarily, and in this paper we will always choose it to be identically zero. Also, Scharlemann allows parameterizing surfaces to contain spherical components. No harm is done to \cite{MSc3} by forbidding them and some simplicity is gained since spherical components have negative index. Lackenby in \cite{L1} has a similar convention.
\end{remark}

We define a parameterizing surface in a banded sutured manifold $(M,\gamma,\beta)$ to be a parameterizing surface $Q$ in $(M,\gamma,\beta - c_\beta)$.

If $Q$ is a parameterizing surface and if $S \subset M$ is a decomposing surface, we say that $S$ and $Q$ are \defn{normalized} if they have been isotoped in a neighborhood of $A(\gamma) \cup T(\gamma)$ to intersect minimally and if no component of $S \cap Q$ is an inessential circle on $Q$. It is clear that if $S$ is $\beta$-taut then $S$ and $Q$ can be normalized without increasing the index of $Q$. Furthermore, it is not difficult to see that if $Q_1, \hdots, Q_n$ are parameterizing surfaces, not necessarily disjoint, then a $\beta$-taut decomposing surface $S$ and $Q_1, \hdots, Q_n$ can be simultaneously normalized by an isotopy of $S$ and each $Q_i$. 

Suppose that $(M,\gamma,\beta) \stackrel{S}{\to} (M',\gamma',\beta')$ is a $\beta$-taut decomposition and that $Q \subset M$ is a parameterizing surface. If $S$ is a conditioned surface normalized with respect to $Q$, we say that the decomposition \defn{respects} $Q$ if $Q \cap M'$ is a parameterizing surface. 

The next lemma is a simple extension of \cite[Section 7]{MSc3}. Recall that $S_k$ denotes the oriented double curve sum of $S$ with $k$ copies of $R(\gamma)$.

\begin{lemma}\label{Lem: Conditioned surfaces respect}
Suppose that $(M,\gamma,\beta) \stackrel{S_k}{\to}(M',\gamma',\beta')$ is $\beta$-taut decomposition with $S$ a conditioned surface and that $Q_1, \hdots, Q_n$ are parameterizing surfaces in $(M,\gamma,\beta)$ such that $S_k$ and each $Q_i$ are normalized. Then for $k$ large enough, the decomposition of $(M,\gamma,\beta)$ using $S_k$ respects each $Q_i$ and the index of each $Q_i$ does not increase under the decomposition. 
\end{lemma}
\begin{proof}
Scharlemann \cite[Lemma 7.5]{MSc3} shows that for each $i$, there exists $m_i \in \N$ such that if $k_i \geq m_i$, and if $S_{k_i}$ is normalized with respect to a parameterizing surface $Q_i$, then $Q_i \cap M'$ is a parameterizing surface with index no larger than $Q_i$. Since for each $k$, $S_k$ can be normalized simultaneously with $Q_1, \hdots, Q_n$, we simply need to choose $k \geq \max(m_1,\hdots, m_n)$.
\end{proof}

Suppose that $S$ is a product disc, product annulus, or disc with boundary in $R(\gamma)$.  We say that a parameterizing surface $Q^c \subset M$ is obtained by \defn{modifying} $Q$ relative to $S$ if $Q^c$ is obtained by completely boundary compressing $Q$ using outermost discs of $S - Q$ bounded by outermost arcs having both endpoints in $R_\pm$, normalizing $Q$ and $S$, and then removing all disc components with boundary completely contained in $R(\gamma)$. Scharlemann proves \cite[Lemma 7.6]{MSc3} that modifying a parameterizing surface does not increase index. Lackenby \cite{L1} points out that if $Q^c$ is compressible so is $Q$. 

\begin{lemma}\label{Lem: Product Respecting}
Suppose that $(M,\gamma,\beta) \stackrel{S}{\to} (M',\gamma',\beta')$ is a $\beta$-taut decomposition with $S$ a product disc, product annulus, or disc with boundary in $R(\gamma)$. Let $Q_1, \hdots, Q_n$ be parameterizing surfaces. Then after replacing each $Q_i$ with $Q^c_i$, each of the surfaces $Q^c_i \cap M'$ is a parameterizing surface in $M'$ with index no larger than the index of $Q_i$.
\end{lemma}
\begin{proof}
This is nearly identical to the proof of Lemma \ref{Lem: Conditioned surfaces respect}, but uses \cite[Lemma 7.6]{MSc3}.
\end{proof}

We say that the decomposition described in Lemma \ref{Lem: Product Respecting} \defn{respects} $Q$.

We now assemble some of the facts we have collected.

\begin{theorem}\label{Thm: Rinsed Decomp Respecting}
Suppose that $(M,\gamma,\beta)$ is a band-taut sutured manifold and that $y \in H_2(M,\boundary M)$ is non-zero. Suppose that $Q_1,\hdots, Q_n$ are parameterizing surfaces in $M$. Then there exists a band-taut decomposition 
\[
(M,\gamma,\beta) \stackrel{S}{\to} (M',\gamma',\beta')
\]
respecting each $Q_i$ with $S$ a rinsed band-decomposing surface representing $\pm y$. 
\end{theorem}
\begin{proof}
By Corollary \ref{Cor: Rinsed Decomp Exist}, there exists a band-taut decomposition
\[
(M,\gamma,\beta) \stackrel{S}{\to} (M',\gamma',\beta')
\]
with $S$ a rinsed band-decomposing surface representing $\pm y$. By Lemma \ref{Lem: Conditioned surfaces respect}, if we replace $S$ with $S_k$ for large enough $k$, we may assume that the decomposition respects each $Q_i$. 
\end{proof}

Similarly we have:
\begin{theorem}\label{Thm: Product Decomp}
Suppose that $(M,\gamma,\beta)$ is a band taut sutured manifold and that there exists an allowable product disc or allowable product annulus in $M'$. Let $Q_1, \hdots, Q_n$ be parameterizing surfaces in $M$. Then there exists an allowable product disc or allowable product annulus $P$, such that, after modifying each $Q_i$, the decomposition given by $P$ is band-taut and respects each $Q_i$. 
\end{theorem}
\begin{proof}
This follows immediately from Lemma \ref{Lem: product disc band taut} and Lemma \ref{Lem: Product Respecting}.
\end{proof}

\section{Sutured Manifold Decompositions and Branched Surfaces}\label{Branched Surface}

In \cite[Construction 4.16]{G3}, Gabai explains how to build a branched surface $B(\mc{H})$ from a sequence of sutured manifold decompositions
\[
\mc{H}: (M,\gamma,\beta) \stackrel{S_1}{\to} (M_1,\gamma_1,\beta_1)\stackrel{S_2}{\to} \hdots \stackrel{S_n}{\to} (M_n,\gamma_n,\beta_n)
\]
Essentially, the branched surface is the union $\bigcup\limits_{i=1}^n S_i$ with the intersections smoothed. We will call $B(\mc{H})$ the branched surface \defn{associated} to the sequence $\mc{H}$.

\begin{lemma}\label{Lem: Branched Surface}
If $\mc{H}$ is a sequence of band-taut sutured manifold decompositions, there is an isotopy of $c_\beta$ (relative to $\boundary c_\beta$) to an arc $a$ such that the closure of $a \cap \inter{M_n}$ is $c_{\beta_n}$ and so that $a - c_{\beta_n}$ is embedded in $\boundary M \cup B(\mc{H})$. Furthermore, there is a proper isotopy of $c_\beta$ in $M$ to $(a \cap B(\mc{H})) \cup c_{\beta_n}$.
\end{lemma}

\begin{proof}
By the definition of ``band-taut'' decomposition, the decomposition 
\[(M_i, \gamma_i, \beta_i) \stackrel{S_{i+1}}{\to} (M_{i+1}, \gamma_{i+1}, \beta_{i+1})\]
defines an isotopy $\phi_i$ in $D_{\beta_i}$ of $c_{\beta_i}$  (relative to its endpoints) to an arc $a_i$ such that the intersection of $a_i$ with the interior of $M_{i+1}$ is the core $c_{\beta_{i+1}}$. If the decomposition is of the form (BT1), then the isotopy moves $c_{\beta_i}$ into $\boundary M \cup B(\mc{H})$.  By Lemma \ref{Lem: Isotoping Core}, the intersection of $a_i$ with $\boundary D_{\beta_i}$ consists of arcs, each joining an endpoint of $c_{\beta_i}$ to an endpoint of $c_{\beta_{i+1}}$. Each of these arcs, if not a single point, intersects $D_{\beta_{i+1}}$ in an arc with one endpoint on $\boundary c_{\beta_{i+1}}$ and the other on a point of $\boundary S_{i+1} \cap \boundary D_{\beta_i}$.

Each $\phi_i$ is also a homotopy of $c_{\beta}$. Their concatanation is a homotopy $\phi$ of $c_\beta$. We desire to show $\phi$ can be homotoped to provide an isotopy $\phi'$ in $D_\beta$ of the arc $c_\beta$ (relative to $\boundary c_\beta$) to an arc $a$ so that $a$ intersects the interior of $M_n$ in $c_{\beta_n}$ and $a - c_{\beta_n}$ is embedded in $B(\mc{H})$.

To that end, suppose that $i$ is the smallest index such that the isotopy of $c_{\beta_i}$ to $c_{\beta_{i+1}}$ makes $c_\beta$ non-embedded. This implies that $a_i$ intersects $a_{i-1}$. The arc $a_{i-1}$ lies in $\boundary D_{\beta_{i-1}}$ and the arc $a_i$ lies in $\boundary D_{\beta_i}$. The boundary of $D_{\beta_i}$ is the union of portions of $\boundary D_{\beta_{i-1}}$ with components of $S_i \cap D_{\beta_{i-1}}$. The arcs $a_{i-1}$ and $a_i$, therefore, intersect in closed intervals lying in $\boundary D_{\beta_{i-1}} \cap \boundary D_{\beta_i}$. There are at most two intervals of overlap and each interval of overlap has one endpoint lying on $\boundary S_i$. These intervals of overlap can each be homotoped to be a point of $\boundary S_i \cap \boundary D_\beta$. This homotopy deforms the concatenation of the isotopy from $c_{\beta_{i-1}}$ to $c_{\beta_i}$ with the isotopy from $c_{\beta_i}$ to $c_{\beta_{i+1}}$ to be an isotopy of $c_\beta$ such that $c_{\beta}$ intersects the interior of $M_{i+1}$ in $c_{\beta_{i+1}}$ and, after the isotopy, $c_{\beta} - c_{\beta_{i+1}}$ is embedded in $B(\mc{H})$. By induction on $i$, we create the desired isotopy of $c_\beta$ to $a$. 

The intersection $a \cap \boundary M$ consists of at most two arcs, each with an endpoint at $\boundary c_\beta$ and with the other endpoint at $\boundary c_{\beta_1}$. There is, therefore, also a proper isotopy of $c_\beta$ to $(a \cap B(\mc{H})) \cup c_{\beta_n}$. 
\end{proof}

\section{Band Taut Hierarchies}\label{Band Taut Hierarchies}

Let $(M,\gamma,\beta)$ be a $\beta$--taut sutured manifold and suppose that $U \subset T(\gamma)$. A $\beta$-taut sutured manifold \defn{hierarchy} (cf. \cite[Definition 2.1]{MSc4}) relative to $U$ is a finite sequence
\[
\mc{H}: (M,\gamma) \stackrel{S_1}{\to} (M_1,\gamma_1) \stackrel{S_2}{\to} \hdots \stackrel{S_n}{\to} (M_n,\gamma_n) 
\]
of $\beta$-taut decompositions for which 
\begin{enumerate}
\item[(i)] each $S_i$ is either a conditioned surface, a product disc, a product annulus whose ends are essential in $R(\gamma_{i-1})$, or a disc whose boundary is $\beta$-essential in $R(\gamma_{i-1})$ and each $S_i$ is disjoint from $U$.
\item[(ii)] $H_2(M_n,\boundary M_n - U) = 0$.
\end{enumerate}
If $U = \nil$, then we simply call it a $\beta$-taut sutured manifold hierarchy.

We say that the hierarchy \defn{respects} a parameterizing surface $Q \subset M$ if each decomposition in $\mc{H}$ respects $Q$. (Implicitly, we assume that $Q$ may be modified by isotopies and $\boundary$-compressions during the decompositions as in Section \ref{Param}.)

Suppose that $(M,\gamma,\beta)$ is a band-taut sutured manifold. A \defn{band-taut hierarchy} for $M$ is a $(\beta - c_\beta)$-taut sutured manifold hierarchy $\mc{H}$ for $(M,\gamma,\beta - c_\beta)$ with each decomposition $(M_{i-1},\gamma_{i-1},\beta_{i-1}) \stackrel{S_i}{\to} (M_i,\gamma_i,\beta_i)$ a band-taut decomposition.

\begin{theorem}\label{Thm: Hierarchies Exist}
Suppose that $(M,\gamma,\beta)$ is a band-taut sutured manifold and that $Q_1, \hdots, Q_n$ are parameterizing surfaces in $M$. Then the following are all true:
\begin{enumerate}
\item there exists a band-taut sutured manifold hierarchy 
\[
\mc{H}: (M,\gamma) \stackrel{S_1}{\to} (M_1,\gamma_1) \stackrel{S_2}{\to} \hdots \stackrel{S_n}{\to} (M_n,\gamma_n) 
\]
for $M$ respecting each $Q_i$. 
\item Each surface $S_i$ is a band-decomposing surface and if $S_i$ is conditioned then it is also rinsed.
\item If $y \in H_2(M,\boundary M)$ is non-zero, $S_1$ may be taken to represent $\pm y$
\item There is a proper isotopy of $c_\beta$ in $M$ to an arc disjoint from $S_1$.
\item Let $B(\mc{H})$ be the branched surface associated to $\mc{H}$. There is an isotopy of $c_\beta$ in $D_\beta$ relative to $\boundary c_\beta$ to an arc $a$ such that the the closure of the arc $a \cap \inter{M}_n$ is $c_{\beta_n}$, the arc $a - c_{\beta_n}$ is embedded in $\boundary M \cup B(\mc{H})$. Furthermore, there is a proper isotopy of $c_\beta$ in $D_\beta$ to an embedded arc in $B(\mc{H})$. 
\end{enumerate}
\end{theorem}
\begin{proof}

Let $S_1$ be the surface provided by Theorem \ref{Thm: Rinsed Decomp Respecting} representing $\pm y$ and giving a band-taut sutured manifold decomposition $(M,\gamma,\beta) \stackrel{S_1}{\to} (M_1,\gamma_1,\beta_1)$ respecting $Q$. Let $Q_1$ be the parameterizing surface in $(M_1,\gamma_1,\beta_1)$ resulting from $Q$.

If $H_2(M_1,\boundary M_1 - U) = 0$, we are done. Otherwise, define $S_2$ according to the instructions below. In the description below, it should always be assumed that if $S$ is chosen at step ($i$), then step ($k$) for all $k > i$ will not be applied. 

\begin{enumerate}

\item If $e_{\beta_1} = \nil$ and if $D_{\beta_1}$ is a boundary parallel disc in $M - (\beta - c_\beta)$, then let $S_2 = c_{\beta_2} = D_{\beta_2} = \nil$. The decomposition by $S_2$ is of the form (BT1).

\item If $(M_1,\gamma_1,\beta_1 - c_{\beta_1})$ contains a cancelling disc or product disc, it contains one disjoint from $D_{\beta_1}$ (Lemma \ref{Lem: Cancelling Discs Disjoint}). Let $S_2$ be either a product disc disjoint from $D_{\beta_1}$ or the frontier of a regular neighborhood of a cancelling disc disjoint from $D_{\beta_1}$. If $S_2$ is the frontier of a cancelling disc, after decomposing along $S_2$, cancel the edge of $\beta_1$ adjacent to the cancelling disc. 

\item If $(M_1,\gamma_1,\beta_1 - c_{\beta_1})$ contains a nonself amalgamating disc, amalgamate an arc component of $\beta_1$ and let $S_2 = \nil$. This does not affect the fact that $(M_1,\gamma_1,\beta_1)$ is a band-taut sutured manifold by \cite[Lemmas 4.3 and 4.4]{MSc3}.

\item If $(M_1,\gamma_1,\beta_1 - c_{\beta_1})$ contains an allowable product disc choose one $S_2$ that is disjoint from $D_{\beta_1}$ (Lemma \ref{Lem: Cancelling Discs Disjoint}). Modify $Q_1$ so that $S_2$ respects $Q_1$. Decomposing along $S_2$ gives a band-taut decomposition by Theorem \ref{Thm: Product Decomp}.

\item If $(M_1,\gamma_1,\beta_1 - c_{\beta_1})$ has an allowable self amalgamating disc, choose one that is disjoint from $D_{\beta_1}$. This is possible by Lemma \ref{Lem: Cancelling Discs Disjoint}. If the associated product annulus has both ends essential in $R(\gamma)$, let $S_2$ be that annulus. Otherwise, let $S_2$ be the disc obtained by isotoping the disc obtained by capping off the annulus with a disc in $R(\gamma)$ so that it is properly embedded in $M$.  Modify $Q_1$ so that $S_2$ respects $Q_1$. 

\item If $(M_1,\gamma_1,\beta_1 - c_{\beta_1})$ has no product discs, cancelling discs, or allowable nonself amalgamating discs, let $S_2$ be the surface obtained by applying Theorem \ref{Thm: Rinsed Decomp Respecting} to an nontrivial element of $H_2(M_2,\boundary M_2 - U)$.
\end{enumerate}

Decompose $(M_1,\gamma_1,\beta_1)$ using $S_2$ to obtain $(M',\gamma'_2, \beta'_2)$. If $S_2$ was a disc with boundary in $R(\gamma_1)$, amalgamate arcs and convert an arc to a suture as in Section \ref{Section: conversion}. By the results of that section, this elimination of nonself amalgamating discs preserves the fact that the resulting sutured manifold $(M_2,\gamma_2,\beta_2)$ is band-taut. Let $Q_2$ be the resulting parameterizing surface in $M_2$. 

If $H_2(M_2,\boundary M_2 - U) = 0$ we are done. Otherwise, a sutured manifold $(M_3,\gamma_3,\beta_3)$ can be obtained from $(M_2,\gamma_2,\beta_2)$ by a method analogous to how we obtained $(M_2,\gamma_2,\beta_2)$ from $(M_1,\gamma_1,\beta_1)$. Repeating this process creates a sequence of band-taut sutured manifold decompositions
\[
\mc{H}: (M,\gamma,\beta) \stackrel{S_1}{\to} (M_1,\gamma_1,\beta_1) \stackrel{S_2}{\to}  (M_2,\gamma_2,\beta_2) \stackrel{S_3}{\to} \hdots
\]
respecting $Q$. 

By the proofs of \cite[Theorem 4.19]{MSc3} and \cite[Theorem 2.5]{MSc4}, the sequence
\[
(M_1,\gamma_1,\beta_1 - c_{\beta_1}) \stackrel{S_2}{\to} (M_2,\gamma_2,\beta_2 - c_{\beta_2}) \stackrel{S_3}{\to} \hdots
 \]
must terminate in $(M_n,\gamma_n,\beta_n - c_{\beta_n})$ with $H_2(M_n,\boundary M_n - U) = 0$. Consequently, $\mc{H}$ is finite. This sequence with all arc cancellations and amalgamations ignored is the desired hierarchy. If $c_{\beta_i} \neq \nil$ but $c_{\beta_{i+1}} = \nil$ then, by the definition of band taut decomposition, $c_{\beta_i}$ can be isotoped in $D_{\beta_i}$ (rel $\boundary c_{\beta_i}$) in $\boundary M_i$. Conclusions (3) and (4) follow from Lemma \ref{Lem: Branched Surface} \end{proof}

Many arguments in sutured manifold theory require showing that a hierarchy remains taut after removing some components of $\beta$. We will need the following theorem, which is a slight generalization of what is stated in \cite{MSc4} (and is implicit in that paper and in \cite{MSc3}).

\begin{theorem}[{\cite[Lemma 2.6]{MSc4}}]\label{Thm: Tautness Up}
Suppose that 
\[
(M,\gamma,\beta)=(M_0,\gamma_0,\beta_0) \stackrel{S_1}{\to} (M_1,\gamma_1,\beta_1) \stackrel{S_2}{\to} \hdots \stackrel{S_n}{\to} (M_n,\gamma_n,\beta_n)
\]
is a sequence of $\beta$-taut sutured manifold decompositions in which
\begin{enumerate}
\item no component of $M$ is a solid torus disjoint from $\gamma \cup \beta$
\item each $S_i$ is either a conditioned surface, a product disc, a product annulus with each boundary component essential in $R(\gamma_{i-1})$, or a disc $D$ such that
\begin{enumerate}
\item $\boundary D \subset R(\gamma_{i-1})$
\item If $\boundary D$ is $\beta$-inessential in $R(\gamma_{i-1})$ then $D$ is disjoint from $\beta$.
\end{enumerate}
\item If a closed component of $S_i$ separates, then it bounds a product region with a closed component of $R(\gamma)$ intersecting $\beta$ in vertical arcs.
\end{enumerate}

Then if $(M_n,\gamma_n,\beta_n)$ is $\beta_n$-taut so is every decomposition in the sequence.
\end{theorem}
\begin{proof}
The only difference between this and what is found in \cite{MSc4} is that we allow closed components of $S_i$ to be parallel to closed components of $R(\gamma)$. Decomposing along such a component creates a sutured manifold equivalent to the original and so if the sutured manifold after the decomposition is $(\beta_{i+1})$-taut, the original must be $(\beta_i)$-taut.
\end{proof}

\begin{remark}
The reason for stating this generalization of \cite[Lemma 2.6]{MSc4} is that in creating a sutured manifold hierarchy that respects a parameterizing surface we may need to decompose along the double curve sum $S_k$ of a conditioned surface $S$ with some number of copies of $R(\gamma)$. If $S$ is disjoint from a closed component of $R(\gamma)$ then some components of $S_k$ will be closed and separating.
\end{remark}
The next corollary is immediate:
\begin{corollary}\label{Cor: Tautness Up}
Suppose that $(M,\gamma,\beta)$ is a band-taut sutured manifold such that no component of $M$ is solid torus disjoint from $(\beta - c_\beta) \cup \gamma$ and that 
\[
\mc{H}: (M,\gamma) \stackrel{S_1}{\to} (M_1,\gamma_1) \stackrel{S_2}{\to} \hdots \stackrel{S_n}{\to} (M_n,\gamma_n) 
\] 
is the band-taut sutured manifold hierarchy given by Theorem \ref{Thm: Hierarchies Exist}. If $(M_n,\gamma_n)$ is $\beta_n - (e_{\beta_n} \cup c_{\beta_n})$-taut, then $(M,\gamma)$ is $\beta - (e_\beta \cup c_\beta)$-taut.
\end{corollary}

Before analyzing the parameterizing surface at the end of the hierarchy, we present one final lemma in this section. Recall that if $b \subset \boundary M$ is a simple closed curve and if $Q \subset M$ is a surface, then a $b$-boundary compressing disc for $Q$ is a disc whose boundary consists of an arc on $Q$ and a sub-arc of $b$. Suppose that $\beta \subset M$ is an edge with both endpoints on $\boundary M$ and that $b$ is a meridian of $\beta$ in the boundary of $M - \inter{\eta}(\beta)$. If $Q$ is a surface in $M - \inter{\eta}(\beta)$, we define a \defn{$\beta$-boundary compressing disc} for $Q$ in $M$ to be a $b$-boundary compressing disc for $Q$ in $M - \inter{\eta}(\beta)$.

The next lemma gives a criterion for determining when a compressing disc or $\beta$-boundary compressing disc for a parameterizing surface at the end of a hierarchy can be pulled back to such a disc for a parameterizing surface in the initial sutured manifold.

\begin{lemma}\label{Lem: compressing up}
Suppose that 
\[
(M,\gamma,\beta) \stackrel{S}{\to} (M',\gamma',\beta') 
\]
is sutured manifold decomposition respecting a parameterizing surface $Q$ with $\beta$ a single arc. Assume that $\mu(Q) \geq 1$. Let $Q'$ be the resulting parameterizing surface in $M'$. Let $\beta'_0$ be a component of $\beta'$. Then if the surface $Q'$ has a compressing disc or $\beta'_0$-boundary compressing disc with interior disjoint from $\beta'$ then the surface $Q$ has a compressing disc or $\beta$-boundary compressing disc.
\end{lemma}
\begin{proof}
Let $D$ be a compressing or $\beta'_0$-boundary compressing disc for $Q'$. Either $Q' = Q - \inter{\eta}(S)$ or $Q' = Q^c - \inter{\eta}(S)$ is obtained by first modifying $Q$ to $Q^c$. If $Q' = Q - \inter{\eta}(S)$, then $Q' \subset Q$ and $D$ is also a compressing or $\beta$-boundary compressing disc for $Q$. We may, therefore, assume that if $Q' = Q^c - \inter{\eta}(S)$, then $Q^c$ has a compressing or $\beta$-boundary compressing disc $E$. By a small isotopy we may assume that $E \cap \boundary M = \nil$.

By the construction of $Q^c$, $Q$ can be obtained from $Q^c$ by tubing $Q^c$ to itself and to discs with boundary in $R(\gamma)$ using tubes that are the frontiers of regular neighborhoods of arcs in $R(\gamma)$. The disc $E$ is easily made disjoint from those tubes by a small isotopy, and so $E$ remains a compressing or $\beta$-boundary compressing disc for $Q$.
\end{proof}

\section{Combinatorics}\label{Combinatorics}
We begin this section with a sequence of lemmas concerning sutured manifolds that are at the end of a band-taut hierarchy. We consider only the situation in which $e_\beta = \beta - c_\beta$. Recall from the definition of banded sutured manifold that $|e_\beta| \leq 2$.

\begin{lemma}\label{Lem: trivial homology}
Suppose that $(M,\gamma,\beta)$ is a connected band-taut sutured manifold with $H_2(M,\boundary M) = 0$ and $\beta - c_\beta = e_\beta$. Assume that $\boundary M \neq \nil$. Then $\boundary M$ is the union of one or two spheres, each component of $\boundary M$ contains exactly one disc component of $R_-$ and exactly one disc component of $R_+$, and one of the following holds:
\begin{enumerate}
\item $e_\beta = \nil$, $|\gamma| = 1$, and $M$ is a 3--ball.
\item $|e_\beta| = 1$, $\boundary M$ is a single sphere, and $e_\beta$ has endpoints in the disc components of $R(\gamma)$. 
\item $|e_\beta| = 2$, $\boundary M$ is a single sphere, one edge $e$ of $e_\beta$ has endpoints in the disc components of $R(\gamma)$ and the other edge of $e_\beta$ has endpoints either in the same disc components, in which case $|\gamma| = 1$, or in the adjacent annulus components of $R(\gamma)$. 
\item $|e_\beta| = 2$, $M = S^2 \times [0,1]$, the edges of $e_\beta$ are fibers in the product structure of $M$. Each component of $\boundary M$ contains a single suture. 
\end{enumerate}
\end{lemma}
\begin{proof}
Since $H_2(M,\boundary M) = 0$, by the ``half lives, half dies'' theorem of algebraic topology, the boundary of $M$ must be the union of spheres. Since $M$ is $e_{\beta}$-irreducible, if there is a component of $\boundary M$ that is disjoint from $e_\beta$, then $\boundary M$ must be that sphere and $|e_\beta| = 0$. Since $(M,\gamma,e_\beta)$ is $e_\beta$-taut, this implies conclusion (1). Assume, therefore, that $|e_\beta| \in \{1,2\}$ and that each component of $\boundary M$ is adjacent to a component of $e_\beta$. 

If $M$ contains a sphere intersecting $e_\beta$ exactly once, then since $R(\gamma)$ is $e_\beta$-incompressible, $\boundary M$ must be the union of two spheres, $|e_\beta| = 1$ and $|\gamma| = 0$. This contradicts the definition of banded sutured manifold. Thus, each component of $\boundary M$ contains at least two endpoints of $e_\beta$. Furthermore, each disc component of $R(\gamma)$ must contain an endpoint of $e_\beta$ since $R(\gamma)$ is $e_\beta$-incompressible. We conclude that $\boundary M$ has no more than two components. 

If $\boundary M$ has two components, then each of them must contain two endpoints of $e_\beta$ and so $|e_\beta| = 2$. In each component of $\boundary M$ the endpoints of $e_\beta$ are contained in disc components of $R(\gamma)$.  Each component of $D_\beta \cap \boundary M$ crosses $\gamma$ exactly once and so each component of $\boundary M$ contains exactly one suture. The frontier of a regular neighborhood of $D_\beta \cup \boundary M$ is a sphere in $M - e_\beta$ which must bound a 3--ball in $M - e_\beta$. Thus, $M = S^2 \times [0,1]$ and the components of $e_\beta$ are fibers. This is conclusion (4).

We may assume, therefore, that $\boundary M$ is a single sphere. Suppose that $R_-$ (say) has two disc components $R_1$ and $R_2$. The discs $R_1$ and $R_2$ must each contain an endpoint $v_1$ and $v_2$, respectively, of $e_\beta$. Since $(M,\gamma)$ is $e_\beta$-taut, each component of $e_\beta$ has one endpoint in $R_-$ and one in $R_+$. Thus, $v_1$ and $v_2$ belong to different components of $e_\beta$. (Consequently, $|e_\beta| = 2$.) Let $w_1$ and $w_2$ be the other endpoints of $e_\beta$ (lying in $R_+$) so that $v_i$ and $w_i$ are endpoints of the same edge of $e_\beta$. 

If $R_+$ has a disc component, then one of $w_1$ or $w_2$ must lie in it. Without loss of generality, suppose it is $w_1$. A component of $\boundary D_\beta \cap \boundary M$ joins $w_1$ to $v_2$ and crosses $\gamma$ once. The union of the disc component of $R_+$ containing $w_1$ with $R_2$ is a sphere and so $\boundary M$ contains more than one component, a contradiction. This implies that if $R_\pm$ contains two discs, then $R_\mp$ cannot contain any. Since $\boundary M$ is a sphere, $R(\gamma)$ contains exactly two discs. Hence, all other components of $R(\gamma)$ are annuli. We see, therefore, that if $|\gamma|$ is even then $R_\pm$ contains two discs and all other components of $R(\gamma)$ are annuli and if $|\gamma|$ is odd then each of $R_-$ and $R_+$ contains a disc and all other components of $R(\gamma)$ are annuli.

If $|e_\beta| = 1$, then since one endpoint of $e_\beta$ is in $R_-$ and the other is in $R_+$ and since each disc component of $R(\gamma)$ contains an endpoint, conclusion (2) holds. Assume, therefore, that $|e_\beta| = 2$. Suppose, for the moment, that some disc component $D$ of $R(\gamma)$ contains two endpoints of $e_\beta$. These endpoints must belong to different components of $e_\beta$. Each component of $\boundary D_\beta \cap \boundary M$ joins endpoints of $e_\beta$ and crosses $\gamma$ once. Thus, the other endpoints of $e_\beta$ are in the component of $R(\gamma)$ adjacent to $D$. This component must, therefore be a disc and so (3) holds. We may assume, therefore, that each disc component of $R(\gamma)$ contains exactly one endpoint of $e_\beta$. 

Suppose that $|\gamma|$ is odd. Let $D_\pm$ be the disc component of $R_\pm$. Each contains an endpoint $v_\pm$ of $e_\beta$. If $v_-$ and $v_+$ do not belong to the same arc of $e_\beta$, then the other endpoint of the arc containing $v_+$ lies in the component of $R(\gamma)$ adjacent to $D_-$, since each component of $D_\beta \cap M$ intersects $\gamma$ exactly once. But this component must lie in $R_+$ and so a component of $e_\beta$ has both endpoints in $R_+$, a contradiction. Thus, $v_-$ and $v_+$ are endpoints of the same component of $e_\beta$, and the fact that each component of $D_\beta \cap \boundary M$ intersects $\gamma$ exactly once immediately implies conclusion (3).

Suppose that $|\gamma|$ is even. Then both disc components of $R(\gamma)$ lie, without loss of generality, in $R_-$. Each contains exactly one endpoint of $e_\beta$. All other components of $R_-$ are annuli disjoint from $\beta$. Thus, $x_{e_\beta}(R_-) = 0$. The surface $R_+$ is the union of annuli, one or two of which contain the two endpoints of $e_\beta \cap R_+$. Thus, $x_{e_\beta}(R_+) = 2$. The union $R_- \cup A(\gamma)$ is a surface with boundary equal to $\boundary R_+$ and homologous to $R_+$ in $H_2(M,\boundary R_+)$. Consequently, $R_+$ is not $x_{e_\beta}$-minimizing, and, therefore, not $e_\beta$-taut. This contradicts our hypotheses. Hence, $|\gamma|$ cannot be even and so each of $R_\pm$ contains a single disc.
\end{proof}

\begin{lemma}\label{Lem: Single Sphere}
Suppose that $(M,\gamma,\beta)$ is a connected band-taut manifold such that $H_2(M,\boundary M) = 0$, $\boundary M$ is connected and non-empty, and $\beta - c_\beta = e_\beta$. Then the number of sutures $|\gamma|$ is odd and there is an edge component $e$ of $e_\beta$ such that $(M,\gamma,e)$ is $e$-taut. Furthermore, if $(M,\gamma)$ is not $\nil$-taut, then either $|\gamma| \geq 3$ or $M$ is a non-trivial rational homology ball. 
\end{lemma}
\begin{proof}
Lemma \ref{Lem: trivial homology} implies that either $(M,\gamma)$ is a 3--ball with a single suture in its boundary or one of the following occurs:
\begin{itemize}
\item $|e_\beta| = 1$, and $R_-$ and $R_+$ each contain a single disc. The intersection of these discs with $e_\beta$ is the endpoints of an edge $e$ of $e_\beta$.
\item $|e_\beta| = 2$, $R_-$ and $R_+$ each contain a single disc. Unless $|\gamma| = 1$, there is a component $e$ of $e_\beta$ such that the intersection of the disc components of $R(\gamma)$ with $e_\beta$ is the endpoints of $e$. If $|\gamma| = 1$, then that intersection contains all the endpoints of $e_\beta$.
\end{itemize}
Let $e$ be an edge of $e_\beta$ having endpoints in the disc components of $R(\gamma)$. Since $e$ does not have both endpoints in $R_\pm$, and since $R(\gamma)$ contains exactly two disc components, $\gamma$ consists of an odd number of parallel sutures on the sphere $\boundary M$. 

We claim that $(M,\gamma,e)$ is $e$-taut. Since $R(\gamma)$ has two disc components and $|\gamma|$ is odd, $R_-$ and $R_+$ are each $e$-minimizing. Suppose, first, that $S$ is an $e$-reducing sphere. Choose $S$ to minimize $|S \cap D_\beta|$. An innermost circle argument shows that $S \cap D_\beta$ is empty, and so $S$ is disjoint from $e_\beta$. Since $M$ is $e_\beta$-taut, $S$ bounds a ball disjoint from $e$, a contradiction. Suppose, therefore, that $S$ is a compressing disc for $R(\gamma) - e$ that is disjoint from $e$. Since $\boundary M$ is a 2-sphere, there is a 2-sphere in $M$ intersecting $e$ a single time. Hence, $M$ contains a non-separating $S^2$, contradicting the assumption that $H_2(M,\boundary M) = 0$. Thus, $(M,\gamma, e)$ is $e$-taut. 

If $M$ is a 3-ball with a single suture in its boundary, then $(M,\gamma)$ is $\nil$-taut. Thus, either $|\gamma| \geq 3$ or $M$ is not a 3-ball. The relative long exact sequence for $H_2(M,\boundary M)$ shows that $H_2(M) = 0$ and that $H_1(M)$ is isomorphic to $H_1(M,\boundary M)$. Duality for manifolds with boundary shows that $H^1(M) = 0$ since $H_2(M,\boundary M) = 0$. The universal coefficient theorem shows that $H^1(M)$ is isomorphic to the direct sum of the free part of $H_1(M)$ and the torsion part of $H_0(M)$. Thus, $H_1(M)$ is finite. This implies that if $M \neq B^3$, then $M$ is a non-trivial rational homology ball, as desired.
\end{proof}

The presence of a parameterizing surface can give us more information.

\begin{lemma}\label{Lem: Lens Space}
Suppose that $(M,\gamma,e)$ is a connected $e$-taut sutured manifold with $e$ an edge. Assume that  $H_2(M,\boundary M) = 0$. Suppose that $Q \subset M$ is a parameterizing surface having no compressing or $e$-boundary compressing disc. If $\mu(Q) \geq 1$, then one of the following holds:
\begin{enumerate}
\item $M$ is a 3--ball, $|\gamma| = 1$ and $e$ is boundary-parallel by a component of $Q$.
\item $M$ is a punctured lens space and $e$ is a core of $M$.
\item $I(Q) \geq 2\mu(Q)$.
\end{enumerate}
\end{lemma}
\begin{proof}
Since $Q$ is a parameterizing surface, no component has negative index. Removing all components of $Q$ that are disjoint from $e$ does not increase index. Since $Q$ is incompressible, no component of $\boundary Q$ is an inessential circle in $\boundary M - e$. Since $\boundary M$ is the union of 2--spheres and since $Q$ has no $e$-boundary compressing disc, each arc component of $\boundary Q \cap \boundary M$ joins distinct endpoints of $e$. Since $\mu(Q) \geq 1$, there is at least one such arc component. Thus, no component of $\boundary Q$ is an essential circle in $\boundary M - e$. Therefore, each component of $\boundary Q \cap \boundary M$ is an arc joining the endpoints of $e$. Isotope $Q$ so as to minimize $\boundary Q \cap \gamma$. This does not increase $I(Q)$. Since $e$ is an edge and $(M,\gamma,e)$ is $e$-taut, $|\gamma|$ must be odd.

\textbf{Case 1:} $|\gamma| = 1$. 

If some component $Q_0$ of $Q$ is a disc intersecting $\gamma$ once, then it is a cancelling disc for $e$. This implies that $(M,\gamma)$ is $\nil$-taut. $M$ is, therefore, a 3--ball and $e$ is boundary parallel by a component of $Q$. 

Suppose, therefore, that some component of  $Q$ is a disc intersecting $\gamma$ more than once (and, therefore, running along $e$ more than once). Compressing the frontier of $\eta(\boundary M \cup e)$ using that disc produces a 2--sphere which must bound a 3--ball. Hence, $M$ is a punctured lens space with core $e$. 

If no component of $Q$ is a disc, then no component of $Q$ has positive euler characteristic, and so $I(Q) \geq \mu(Q) + |\boundary Q \cap \gamma| = 2\mu(Q)$.

\textbf{Case 2:} $|\gamma| \geq 3$.

There are at most $\mu(Q)$ components of $Q$ and so $-2\chi(Q) \geq -2\mu(Q)$. We have, therefore, $I(Q) \geq \mu(Q) + |\boundary Q \cap \gamma| - 2\mu(Q)$. Since $|\gamma| \geq 3$ and since all sutures are parallel, each arc of $\boundary Q \cap \boundary M$ intersects $\gamma$ at least 3 times. Thus,
\[
I(Q) \geq \mu(Q) + 3\mu(Q) - 2\mu(Q) = 2\mu(Q)
\]
as desired.
\end{proof}

The next theorem is the key result of the paper. It applies the combinatorics of the previous lemmas to the last term of a band-taut hierarchy.

\begin{theorem}\label{Thm: Main Theorem}
Suppose that $(M,\gamma,\beta)$ is a band taut sutured manifold with $e_\beta = \beta - c_\beta$. Assume that $e_\beta$ has components $e_1$ and $e_2$. Let $Q_1$ and $Q_2$ be parameterizing surfaces in $(M,\gamma,e_\beta)$ with $Q_1 \cap e_2 = Q_2 \cap e_1 = \nil$. We allow the possibility that $e_i = Q_i = \nil$ for $i \in\{1,2\}$.

Then one of the following occurs:
\begin{enumerate}
\item Some $Q_i$ has a compressing or $e_i$-boundary compressing disc in $(M,\gamma,e_i)$.
\item $|e_\beta| = 2$ and $M$ contains an $S^2$ intersecting each edge of $e_\beta$ exactly once.
\item For some $i$, $(M,e_i) = (M'_0,\beta'_0) \# (M'_1,\beta'_1)$ where $M'_1$ is a lens space and $\beta'_1$ is a core of $M'_1$.
\item $(M,\gamma)$ is $\nil$-taut. The arc $c_\beta$ can be properly isotoped onto a branched surface $B(\mc{H})$ associated to a taut sutured manifold hierarchy $\mc{H}$ for $M$. Also, a proper isotopy of $c_\beta$ in $M$ takes $c_\beta$ to an arc disjoint from the first decomposing surface of $\mc{H}$. That first decomposing surface can be taken to represent $\pm y$ for any given non-zero $y \in H_2(M,\boundary M)$.
\item Either 
\[I(Q_1) \geq 2\mu(Q_1) \text{ or } I(Q_2) \geq 2\mu(Q_2).\]
\end{enumerate}
\end{theorem}
\begin{proof}
By Theorem \ref{Thm: Hierarchies Exist}, there exists a band-taut hierarchy
\[
\mc{H}: (M,\gamma,\beta) \stackrel{S_1}{\to} \hdots \stackrel{S_n}{\to} (M_n,\gamma_n,\beta_n)
\]
respecting $Q_1$ and $Q_2$ with $S_1$ representing $\pm y$. By that theorem, there is a proper isotopy of $c_\beta$ in $M$ to an arc disjoint from $S_1$. Let $c_{\beta_n}$ be the core of the band in $M_n$. By Theorem \ref{Thm: Hierarchies Exist}, there is an isotopy of $c_\beta$ so that $c_\beta - c_{\beta_n}$ is embedded in the union of $\boundary M$ with the branched surface $B(\mc{H})$. At each stage of the hierarchy, each component of $D_{\beta_i} - S_i$ not containing $c_{\beta_{i+1}}$ is a cancelling disc, product disc, or amalgamating disc (Lemma \ref{Lem: Defining band taut decomp}), and the hierarchy is constructed so as to eliminate all such discs. Thus, we may assume that each component of $\beta_n - (e_{\beta_n} \cup c_{\beta_n})$ is an arc in a 3-ball component of $M_n$ having a single suture in its boundary; that 3-ball is disjoint from all other components of $\beta_n$. Deleting such arc components preserves the $(\beta_n - c_{\beta_n})$-tautness of $(M_n,\gamma_n)$. Henceforth, we ignore such components.

Either conclusion (1) of our theorem occurs, or by Lemma \ref{Lem: compressing up}, $Q_i$ does not have a compressing or $e_i$-boundary compressing disc in $(M_n,\gamma_n, \beta_n - (c_{\beta_n} \cup e_j)$ (with $j \neq i$). We assume that $Q_i$ does not have such a disc.

The manifold $M_n$ has $H_2(M_n,\boundary M_n) = 0$. Let $M'$ denote the component of $M_n$ containing $c_{\beta_n}$. Let $\gamma' = \gamma \cap M'$. We have $H_2(M',\boundary M') = 0$. By Lemma \ref{Lem: trivial homology}, $\boundary M_n$ is the union of one or two spheres and one of the following holds:
\begin{enumerate}
\item[(a)] $e_{\beta_n} = \nil$, $|\gamma'| = 1$, and $M'$ is a 3--ball.
\item[(b)] $|e_{\beta_n}| = 1$, $\boundary M'$ is a single sphere, each of $R_-$ and $R_+$ contains a single disc, and $e_{\beta_n}$ has endpoints in the disc components of $R(\gamma)$. 
\item[(c)] $|e_{\beta_n}| = 2$, $\boundary M'$ is a single sphere, one edge $e$ of $e_{\beta_n}$ has endpoints in the disc components of $R(\gamma')$ and the other edge of $e_{\beta_n}$ has endpoints either in the same disc components, in which case $|\gamma| = 1$ or in the adjacent annulus components of $R(\gamma')$. 
\item[(d)] $|e_{\beta_n}| = 2$, $M' = S^2 \times [0,1]$, the edges of $e_{\beta_n}$ are fibers in a product structure of $M'$. Each component of $\boundary M'$ contains a single suture. 
\end{enumerate}

If (d) occurs then we have conclusion (2) of our theorem. Assume, therefore, that neither (1) nor (2) of our theorem occur.   

If (a) occurs, then $(M_n,\gamma_n)$ is $\nil$-taut and by Corollary \ref{Cor: Tautness Up}, the sequence $\mc{H}$ is $\nil$-taut. The disc $D_{\beta_n}$ is isotopic into $\boundary M'$ and so the hierarchy $\mc{H}$ can be extended by a decomposition satisfying (BT1) with empty decomposing surface. This gives conclusion (4). 

Assume, therefore, that $|e_{\beta_n}| \geq 1$. By Lemma \ref{Lem: Single Sphere}, $|\gamma'|$ is odd and there exists an edge $e$ of $e_{\beta_n}$ such that $e$ has both endpoints in disc components of $R(\gamma')$ and $(M',\gamma',e)$ is $e$-taut. If $M'$ is a 3--ball and if $e$ is boundary-parallel then, once again, we have conclusion (4). Assume, therefore, that conclusion (4) does not occur.

The edge $e$ is a subarc of $e_i$ for some $i \in \{1,2\}$. Let $Q'_i$ be the parameterizing surface in $M'$ resulting from $Q_i$. By hypothesis, $\mu(Q'_i) \geq 1$ and $Q'_i$ does not have any compressing or $e_i$-boundary compressing discs. By Lemma \ref{Lem: Lens Space}, one of the following occurs:
\begin{enumerate}
\item[(i)] $M'$ is a punctured lens space and $e$ is a core of $M'$.
\item[(ii)] $I(Q'_i) \geq 2\mu(Q'_i)$
\end{enumerate}

If (i) happens then we have conclusion (3) of our theorem. If (ii) happens, then using the facts that $I(Q_i) \geq I(Q'_i)$ and $\mu(Q_i) = \mu(Q'_i)$ we have $I(Q_i) \geq 2\mu(Q_i)$. This is conclusion (5) of our theorem.
\end{proof}

\section{From Arc-Taut to Band-Taut}\label{arc taut to band taut}
We begin this section by constructing a band taut sutured manifold from an arc-taut sutured manifold (that is, a $\beta$-taut sutured manifold where $\beta$ is an arc). 

Let $(M,\gamma,\beta_1)$ be a $\beta_1$-taut sutured manifold with $\beta_1$ an edge having endpoints in components of $R(\gamma)$ with boundary. Let $c_{\beta}$ be obtained by isotoping the endpoints of $\beta_1$ into components of $A_-$ and $A_+$ adjacent to the components of $R(\gamma)$ containing the endpoints of $\beta_1$. Let $\beta_2$ be the arc obtained by continuing to isotope $c_\beta$ so that its endpoints are moved across $\gamma$ and into $R(\gamma)$. Let $D_\beta$ be the disc of parallelism between $\beta_1$ and $\beta_2$ that contains $c_\beta$. Let $\beta = \beta_1 \cup c_\beta \cup \beta_2$. We call $(M,\gamma,\beta)$ the \defn{associated banded sutured manifold}.

\begin{lemma}\label{Lem: Arc Taut to Band Taut}
If $(M,\gamma,\beta_1)$ is a $\beta_1$-taut sutured manifold with $\beta_1$ an edge, then $(M,\gamma,\beta)$ is a band-taut sutured manifold.
\end{lemma}
\begin{proof}
Without loss of generality, we may assume that $M$ is connected. Recall that $e_\beta = \beta_1 \cup \beta_2$. We desire to show that $(M,\gamma,e_\beta)$ is $e_\beta$-taut. Clearly, since $M - \beta_1$ is irreducible, $M - e_\beta$ is irreducible. Since $e_\beta$ is disjoint from $T(\gamma)$, $T(\gamma)$ is taut. It remains to show that $R_\pm$ is $e_\beta$-taut.

Let $S$ be a $e_\beta$-taut surface with $\boundary S = \boundary R_\pm$ and $[S,\boundary S] = [R_\pm, \boundary R_\pm]$ in $H_2(M,\boundary R_\pm)$. Out of all such surfaces, choose $S$ to intersect $D_\beta$ minimally. 

Since $S$ is $e_\beta$-taut and since $D_\beta$ is a disc, no component of $S \cap D_\beta$ is a circle or an arc with both endpoints on the same component of $e_\beta$. Since $\boundary S = \boundary R_\pm$, the intersection $S \cap D_\beta$ contains exactly two arcs having an endpoint on $\boundary M$. Since $S$ and $R_\pm$ are homologous, the algebraic intersection number of each surface with each component of $e_\beta$ is the same. Since $S$ is $e_\beta$-taut, the geometric intersection number of $S$ with each component of $e_\beta$ equals the absolute value of the algebraic intersection number. Consequently, $S$ intersects each component of $e_\beta$ exactly once. This implies that $S \cap D_\beta$ consists exactly of two arcs each joining $\boundary M$ to $e_\beta$ and $S$ intersects both components of $e_\beta$. 

Suppose, for a moment, that some disc component $R_1$ of $R_\pm$ is disjoint from $\beta_1$ but not from $e_\beta$. Let $R_2$ be the component of $R_\mp$ adjacent to $R_1$. Since $R_1$ is adjacent to $\beta_2$, $R_2$ must be adjacent to $\beta_1$. Consequently, $R_1$ is a $\beta_1$-compressing disc for $R_2$. This contradicts the fact that $R_2$ is $\beta_1$-incompressible. We conclude that no component of $R_\pm$ is a disc disjoint from $\beta_1$ but not from $\beta_2$. Consequently, 
\[
x_{e_\beta}(R_\pm) = x_{\beta_1}(R_\pm) + 1.
\]

Without loss of generality, we may assume that $S$ contains no sphere component disjoint from $e_\beta$. Thus, if $S_0$ is a component of $S$, then either $x_{e_\beta}(S_0) = -\chi(S_0) + |S_0 \cap e_\beta|$ or $S_0$ is a disc disjoint from $e_\beta$. Suppose that $S_0$ is a disc disjoint from $e_\beta$ and let $R$ be the component of $R_\pm$ with $\boundary S \subset \boundary R$. Since $R$ is $\beta_1$-incompressible, $R$ must be a disc disjoint from $\beta_1$. By the previous paragraph, $R$  is also disjoint from $\beta_2$. This implies that the component of $\boundary M$ containing $R$ is a 2-sphere disjoint from $e_\beta$ and containing a single suture. Since $M$ is $\beta_1$ irreducible, this implies that $M$ is a 3-ball disjoint from $\beta_1$ and having a single suture in its boundary, a contradiction. Thus, no component of $S$ is a disc disjoint from $e_\beta$ and $x_{e_\beta}(S) = -\chi(S) + |S \cap e_\beta|$. 

Similarly, if $S_0$ is a component of $S$ then either $x_{\beta_1}(S_0) = -\chi(S_0) + |S_0 \cap \beta_1|$ or $S_0$ is a disc disjoint from $\beta_1$. The component of $R_\pm$ containing $\boundary S_0$ is $\beta_1$-incompressible and so must be a disc disjoint from $\beta_1$. As before, this implies that $M$ is a 3-ball disjoint from $\beta_1$ with a single suture in its boundary. This contradicts our hypotheses and so $x_{\beta_1}(S) = -\chi(S) + |S \cap \beta_1|$. Consequently,
\[
x_{e_\beta}(S) = x_{\beta_1}(S) + 1.
\]

Since $R_\pm$ is $\beta_1$-minimizing, we have
\[
x_{\beta_1}(R_\pm) \leq x_{\beta_1}(S).
\]

Hence,
\[
x_{e_\beta}(R_\pm) - 1 \leq x_{e_{\beta}}(S) - 1,
\]
and so
\[
x_{e_\beta}(R_\pm) \leq x_{e_{\beta}}(S).
\]
Since $S$ is $e_{\beta}$-minimizing, $R_\pm$ must be as well.

If $R_\pm$ were $e_\beta$-compressible by a compressing disc $D$, the boundary of $D$ would have to be $\beta_1$-inessential in $R_\pm$. Since $\beta_2$ has only one endpoint in $R_\pm$, the union of $D$ with a disc contained in $R_\pm$ produces a sphere intersecting $\beta_2$ exactly once. Since $\beta_1$ and $\beta_2$ are parallel, there is a sphere intersecting $\beta_1$ exactly once transversally. The components of $R(\gamma)$ containing the endpoints of $\beta_1$ are, therefore, $\beta_1$-compressible, a contradiction. Thus $R_\pm$ is $e_\beta$-incompressible and so $(M,\gamma,\beta)$ is band-taut.

\end{proof}

If $(M,\gamma,\beta_1)$ has a parameterizing surface $Q_1$, the isotopy of $\beta_1$ to $\beta_2$ gives an isotopy of $Q_1$ to a parameterizing surface $Q_2$ for $(M,\gamma,\beta_2)$. The next two results give conditions guaranteeing the existence of such an isotopy that does not increase the index of the parameterizing surface. First, we define some notation for the statement of the lemmas.

Let $v_\pm$ be the endpoints of $\beta_1$. Let $\alpha_\pm$ be the path from $v_\pm$ to  the endpoints of $\beta_2$ defined by the isotopy of $\beta_1$ to $\beta_2$. Let $\gamma_\pm$ be the components of $\gamma$ intersecting $\alpha_\pm$. Let $n_\pm$ be the number of arc components of $\boundary Q_1$ in a neighborhood of $v_\pm$. Some arc components may belong to edges of $\boundary Q_1 \cap R(\gamma)$ parallel to $\alpha_\pm$. Let $m_\pm$ be the number of those arcs plus the number of circle components of $\boundary Q \cap \gamma$ parallel to $\gamma_\pm$. 

\begin{lemma}\label{Lem: Counting and Isotoping}
Assume that any component of $\boundary Q_1 \cap R(\gamma)$ intersecting $\alpha_\pm$ is a circle parallel to $\gamma_\pm$. Then there is an isotopy of $Q_1$ to a parameterizing surface $Q_2$ for $(M,\gamma,\beta_2)$ so that
\[
I(Q_2) = I(Q_1) + (n_- + n_+) - 2(m_- + m_+).
\]
\end{lemma}
\begin{proof}
Each arc component of $\boundary Q_1 \cap R_\pm$ contributing to $m_\pm$ can be isotoped to lie entirely in $R_\mp$. Each other arc component of $\boundary Q_1 \cap R_\pm$ after the isotopy of $Q_1$ to $Q_2$ crosses $\gamma$ an additional time. Any component of $\boundary Q \cap R(\gamma)$ intersecting $\alpha_\pm$ can be isotoped across $A(\gamma)$ without changing the index of $Q$, since such a component is hypothesized to be parallel to $\gamma_\pm$.
\end{proof}

Figure \ref{Lem: IsotopingEndpoint} shows an example of an isotopy which decreases index by 1.

\begin{center}
\begin{figure}[ht]
\scalebox{.5}{\includegraphics{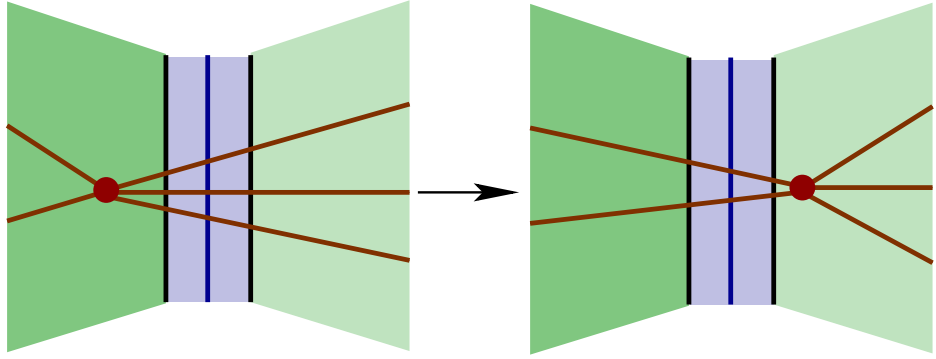}}
\caption{An example with $n_\pm = 5$ and $m_\pm = 3$. The isotopy of the endpoint $v_\pm$ across $\gamma$ reduces the index of the parameterizing surface by 1.}
\label{Lem: IsotopingEndpoint}
\end{figure}
\end{center}

\begin{corollary}\label{Cor: Annular R}
Suppose that $(M,\gamma,\beta_1)$ is a $\beta_1$-taut sutured manifold with $\beta_1$ an arc having endpoints on annular components of $R(\gamma)$. Suppose also that $Q_1$ is a parameterizing surface with $\mu(Q_1) \geq 1$ and that those annular components do not contain any inessential arc or circle of $\boundary Q_1 \cap R(\gamma)$. Let $(M,\gamma,\beta)$ be an associated banded sutured manifold and let $Q_2$ be a parameterizing surface in $(M,\gamma,\beta_2)$ isotopic to $Q_1$. Then $\beta_2$ and $Q_2$ can be chosen so that $I(Q_1) \geq I(Q_2)$ and $Q_1 \cap \beta_2 = Q_2 \cap \beta_1 = \nil$.
\end{corollary}
\begin{proof}
Let $\rho_\pm$ be the components of $R_\pm$ containing the endpoints of $\beta_1$. The surfaces $\rho_\pm - \boundary \beta_1$ are thrice punctured spheres. Let $R = \rho_\pm - \boundary \beta_1$. By hypothesis, each arc of $\boundary Q_1 \cap R$ is an essential arc.  In particular, $\boundary Q_1 \cap R$ has at most one isotopy class of arcs with both endpoints on a single component of $\boundary R(\gamma)$. Choose paths $\alpha_\pm$ from $\boundary \beta_1$ to $A_\pm$ disjoint from any arcs with both endpoints on a single component of $\boundary R(\gamma)$. If there are no arcs with both endpoints on a single component of $\boundary R(\gamma)$, then choose $\alpha_\pm$ to join $\boundary \beta_1$ to the component of $\boundary \rho_\pm$ containing the greatest number of endpoints of $\boundary Q_1 \cap \rho_\pm$. Any arc of $\boundary Q_1 \cap R$ having both endpoints at $\boundary \beta_1$ forms a loop parallel to both components of $\boundary \rho_\pm \cap \boundary R(\gamma)$. Hence, we have satisfied the hypotheses of Lemma \ref{Lem: Counting and Isotoping}. In the notation of that lemma, we have $2m_\pm \geq n_\pm$. Thus, $I(Q_1) \geq I(Q_2)$. A small isotopy makes $Q_1$ disjoint from $\beta_2$ and $Q_2$ disjoint from $\beta_1$.
\end{proof}

We can now use Theorem \ref{Thm: Main Theorem} to obtain a theorem for arc taut sutured manifolds where the arc has endpoints in annulus components of $R(\gamma)$.
\begin{theorem}\label{Thm: Main Theorem A}
Suppose that $(M,\gamma,\beta)$ is a $\beta$-taut sutured manifold with $\beta$ a single edge. Let $Q$ be a parameterizing surface in $M$ with $\mu(Q) \geq 1$. Assume that the endpoints of $\beta$ lie in annulus components $\rho_\pm$ of $R_\pm$ and that no arc or circle of $\boundary Q \cap \rho_\pm$ is inessential. Then one of the following is true:
\begin{enumerate}
\item $Q$ has a compressing or $\beta$-boundary compressing disc.
\item $(M,\beta) = (M'_0,\beta'_0) \# (M'_1,\beta'_1)$ where $M'_1$ is a lens space and $\beta'_1$ is a core of $M'_1$.
\item $(M,\gamma)$ is $\nil$-taut. The arc $\beta$ can be isotoped relative to its endpoints to be embedded on the union of $\boundary M$ with a branched surface $B(\mc{H})$ associated to a taut sutured manifold hierarchy $\mc{H}$ for $M$. Furthermore, there is a proper isotopy of $\beta$ in $M$ to an arc disjoint from the first decomposing surface of $\mc{H}$. That first decomposing surface can be taken to represent $\pm y$ for any given non-zero $y \in H_2(M,\boundary M)$.
\item 
\[
I(Q) \geq 2\mu(Q)
\]
\end{enumerate}
\end{theorem}
\begin{proof}
Let $e_1 = \beta$. By Corollary \ref{Cor: Annular R}, the endpoints of $\beta_1$ can be isotoped across $A(\gamma)$ to create an arc $e_2$ and an associated banded sutured manifold $(M,\gamma,\wihat{\beta})$. By Lemma \ref{Lem: Arc Taut to Band Taut}, this sutured manifold is band-taut. By Corollary \ref{Cor: Annular R}, the isotopy can be chosen so that $Q = Q_1$ is isotoped to a surface $Q_2$ disjoint from $e_1$ such that $I(Q_2) \leq I(Q_1)$. By a small isotopy, we can make $Q_1 \cap e_2 = Q_2 \cap e_1 = \nil$. (The surfaces $Q_1$ and $Q_2$ may intersect.) By Theorem \ref{Thm: Main Theorem}, one of the following happens:

\begin{enumerate}
\item[(a)] Some $Q_i$ has a compressing or $e_i$-boundary compressing disc in $(M,\gamma,e_i)$.
\item[(b)] $M$ contains an $S^2$ intersecting each of $e_1$ and $e_2$ exactly once.
\item[(c)] For some $i$, $(M,e_i) = (M'_0,\beta'_0) \# (M'_1,\beta'_1)$ where $M'_1$ is a lens space and $\beta'_1$ is a core of a genus one Heegaard splitting of $M'_1$.
\item[(d)] $(M,\gamma)$ is $\nil$-taut. The arc $c_{\wihat{\beta}}$ can be isotoped relative to its endpoints to be embedded on the branched surface associated to a taut sutured manifold hierarchy for $M$. Furthermore, there is a proper isotopy of $c_{\wihat{\beta}}$ in $M$ to an arc disjoint from the first decomposing surface of the hierarchy. That first decomposing surface can be taken to represent $\pm y$ for any given non-zero $y \in H_2(M,\boundary M)$.
\item[(e)] Either 
\[I(Q_1) \geq 2\mu(Q_1) \text{ or } I(Q_2) \geq 2\mu(Q_2).\]
\end{enumerate}

Since each $(e_i,Q_i)$ is isotopic to $(\beta,Q)$, possibility (a) implies conclusion (1) of our theorem. Possibility (b) cannot occur since that would imply that there was a $\beta$-compressing disc for $R(\gamma)$. Possibility (c) implies Conclusion (2), since $e_i$ is isotopic to $\beta$. Possibility (d) implies Conclusion (3). Possibility (e) implies conclusion (4) since $I(Q) = I(Q_1) \geq I(Q_2)$ and $\mu(Q) = \mu(Q_2) = \mu(Q_1)$.
\end{proof}

We can now prove Theorem \ref{Thm: Main Theorem C} for the case when the components of $R(\gamma)$ adjacent to $b$ are thrice-punctured spheres. It is really only a slight rephrasing of Theorem \ref{Thm: Main Theorem A}.

\begin{theorem}\label{Thm: Main Theorem B}
Suppose that $(N,\gamma)$ is a taut sutured manifold and that $b \subset \gamma$ is a curve adjacent to thrice-punctured sphere components of $R(\gamma)$. Let $Q$ be a parameterizing surface in $N$ with $|Q \cap b| \geq 1$ and with the property that the intersection of $Q$ with the components of $R(\gamma)$ adjacent to $b$ contains no inessential arcs or circles. Let $\beta$ be the cocore in $N[b]$ of a 2-handle attached along $b$. Then one of the following is true:
\begin{enumerate}
\item $Q$ has a compressing or $b$-boundary compressing disc.
\item $(N[b],\beta) = (M'_0,\beta'_0) \# (M'_1,\beta'_1)$ where $M'_1$ is a lens space and $\beta'_1$ is a core of a genus one Heegaard splitting of $M'_1$.
\item $(N[b],\gamma-b)$ is $\nil$-taut. The arc $\beta$ can be properly isotoped  to be embedded on a branched surface $B(\mc{H})$ associated to a taut sutured manifold hierarchy $\mc{H}$ for $N[b]$. There is also a proper isotopy of $\beta$ in $N[b]$ to an arc disjoint from the first decomposing surface of $\mc{H}$. That first decomposing surface can be taken to represent $\pm y$ for any given non-zero $y \in H_2(N[b],\boundary N[b])$.
\item 
\[
-2\chi(Q) + |Q \cap \gamma| \geq 2|Q \cap b|.
\]
\end{enumerate}
\end{theorem}
\begin{proof}
Let $M = N[b]$. Convert the suture $b$ to an arc $\beta$. Since $(N,\gamma)$ is $\nil$-taut, $(M,\gamma-b, \beta)$ is $\beta$-taut. The theorem then follows immediately from Theorem \ref{Thm: Main Theorem A}.
\end{proof}

\section{Separating Sutures on Genus Two Surfaces}\label{Punctured Tori}
In this section, we prove Theorem \ref{Thm: Main Theorem C} for the case when $b$ is adjacent to once-punctured tori. The key idea is to create a band-taut sutured manifold by viewing a certain decomposition of the original sutured manifold in three different ways.

We say that a sutured manifold $(M,\gamma, \beta)$ is \defn{almost taut} if it satisfies (T1), (T2) from Section \ref{Sec: SM} and also:
\begin{enumerate}
\item[(AT)] $\beta$ is a single edge and either has both endpoints in distinct components of $T(\gamma)$ or has both endpoints in distinct components of $A(\gamma)$.
\end{enumerate}

The strategy is to begin with an arc-taut sutured manifold $M_+ = (M,\gamma,\beta_+)$ where $\beta_+$ is an arc having endpoints in distinct torus components of $R(\gamma)$. Convert it to an almost taut sutured manifold $M_0 = (M,\gamma, c_\beta)$ where $c_\beta$ has endpoints in distinct torus components of $T(\gamma)$, produce a so-called ``almost-taut decomposition'' of $M_0$ resulting in an almost taut sutured manifold $M'_0 = (M',\gamma', c'_\beta)$, convert $M'_0$ to a band-taut sutured manifold $(M',\gamma',\beta')$ and then appeal to Theorem \ref{Thm: Main Theorem}. Along the way we will also have to analyze the behaviour of parameterizing surfaces.

We establish the following notation:

Let $M_+ = (M,\gamma,\beta_+)$ be a sutured manifold, with $\beta_+$ an arc having endpoints in torus components $T_- \subset R_-(\gamma)$ and $T_+ \subset R_+(\gamma)$. Let $M_- = (M,\gamma,\beta_-)$ be the sutured manifold resulting from moving $T_-$ into $R_+$, moving $T_+$ into $R_-$ and performing a small isotopy of $\beta_+$ to an arc $\beta_-$ disjoint from $\beta_+$. Let $M_0 = (M,\gamma,c_\beta)$ be the sutured manifold resulting from moving $T = T_- \cup T_+$ into $T(\gamma)$ and performing a small isotopy of $\beta_+$ to an arc $c_\beta$ that is disjoint from $\beta_+ \cup \beta_-$. 

\subsection{Preliminary tautness results.}
The next lemma is straightforward to prove, and so we omit the proof.

\begin{lemma}\label{Lem: T implies AT}
If $M_+$ is $\beta_+$-taut, then $M_0$ is almost taut.
\end{lemma}

Now suppose that we are given an almost taut sutured manifold $M'_0 = (M',\gamma', c_{\beta'})$ with the endpoints of $c_{\beta'}$ in $A(\gamma)$. We create a banded sutured manifold $(M',\gamma', \beta')$ as follows. Isotope the endpoints of $c_{\beta'}$ out of $A(\gamma')$ and into $R(\gamma')$ so that one endpoint lies in $R_-$ and the other in $R_+$. (We require that once the endpoints leave $A(\gamma')$ they do not reenter it during the isotopy.) Since the endpoints of $c_{\beta'}$ lie in distinct components of $A(\gamma')$, up to ambient isotopy of $M'$ (relative to $A(\gamma')$) there are two ways of isotoping $c_{\beta'}$ so that the endpoints lie in $R(\gamma')$. Let $M'_- = (M',\gamma',\beta'_-)$ and $M'_+ = (M',\gamma',\beta'_+)$ denote the two ways of doing this. Perform the isotopies so that $c_{\beta'}$, $\beta'_+$, and $\beta'_-$ are pairwise disjoint. Let $\beta'$ denote their union, and let $D_{\beta'}$ be an (embedded) disc of parallelism between $\beta'_-$ and $\beta'_+$ that contains $c_{\beta'}$ in its interior. Then $(M',\gamma',\beta')$ is a banded sutured manifold. We say that it is a banded sutured manifold \defn{derived} from $M'_0$. The next lemma gives criteria guaranteeing that the derived sutured manifold is band-taut.

\begin{lemma}\label{Lem: AT implies BT}
Suppose that $(M',\gamma',c_{\beta'})$ is a $c_{\beta'}$-almost taut connected sutured manifold and that $(M',\gamma',\beta')$ is a derived banded sutured manifold. Suppose that no sphere in $M'$ intersects $c_{\beta'}$ exactly once transversally. If no component of $R(\gamma')$ containing an endpoint of $e_{\beta'} = \beta'_- \cup \beta'_+$ is a disc and if $\chi(R_-) = \chi(R_+)$, then $(M',\gamma',\beta')$ is band taut.
\end{lemma}
\begin{proof}
Since each component of $e_{\beta'}$ is isotopic to $c_{\beta'}$ and since no sphere separates the components of $e_{\beta'}$, $(M',\gamma',e_{\beta'})$ is $e_{\beta'}$-irreducible.

Suppose that $R_\pm$ is $e_{\beta'}$-compressible by a disc $D$. Since $R_\pm$ is $c_{\beta'}$-incompressible, the boundary of $D$ bounds a disc $D' \subset R_\pm$ containing one or two endpoints of $e_{\beta'}$. If it contains two endpoints, they must be endpoints of different components of $e_{\beta'}$. Then $D \cup D'$ is a sphere in $M'$ intersecting an edge of $e_{\beta'}$ in a single point. Since each edge of $e_{\beta'}$ is isotopic to $c_{\beta'}$, there is a sphere in $M'$ intersecting $c_{\beta'}$ in a single point, contrary to hypothesis.

Let $S$ be a surface representing $[R_\pm,\boundary R_\pm]$ in $H_2(M',\boundary R_\pm)$ such that:
\begin{itemize}
\item $S$ is $e_{\beta'}$-incompressible
\item $S$ intersects each edge of $e_{\beta'}$ always with the same sign.
\end{itemize}
We wish to show that $x_{e_{\beta'}}(R_\pm) \leq x_{e_{\beta'}}(S)$.

Isotope $S$, relative to $\boundary S$, to minimize the pair $(|D_{\beta'} \cap S|, |c_{\beta'} \cap S|)$ lexicographically. An innermost disc argument shows that $S$ intersects $D_{\beta'}$ in arcs only. An outermost arc argument shows that each of these arcs has an endpoint on $\boundary M$ or joins $\beta'_-$ to $\beta'_+$. Since $S$ represents $[R_\pm, \boundary R_\pm]$, the algebraic intersection of $S$ with each component of $e_{\beta'}$ is $\pm 1$ and the algebraic intersection of $S$ with $c_{\beta'}$ is zero. The absolute value of the algebraic intersection of $S$ with each edge of $e_{\beta'}$ is equal to the geometric intersection number. Since $|\boundary S \cap D_{\beta'}| = 2$, there are two arcs in $S \cap D_{\beta'}$. Since the algebraic intersection number of $S$ with each component of $e_{\beta'}$ is $\pm 1$, each of $\beta'_-$ to $\beta'_+$ is incident to exactly one arc of $S \cap D_{\beta'}$. If an arc of $S \cap D_{\beta'}$ joins $\beta'_-$ to $\beta'_+$, then $S$ would have algebraic intersection number $\pm 1$ with $c_{\beta'}$. This contradicts the fact that $(S,\boundary S)$ is homologous to $(R_\pm, \boundary R_\pm)$. Thus, neither arc joins $\beta'_-$ to $\beta'_+$. Similarly, since $S \cap D_{\beta'}$ contains two arcs and since each of $\beta'_-$ and $\beta'_+$ intersects an arc and since they don't intersect the same arc, each arc of $S \cap D_{\beta'}$ joins $\boundary M'$ to $e_{\beta'}$. Since $S$ has zero algebraic intersection with $c_{\beta'}$, as in Figure \ref{Fig: ATimpliesBT}, either these arcs are both disjoint from $c_{\beta'}$ or they each intersect $c_{\beta'}$ exactly once.

\begin{figure}[ht!] 
\labellist \small\hair 2pt 
\pinlabel {$c_{\beta'}$} [l] at 186 169
\pinlabel {$c_{\beta'}$} [l] at 507 194
\pinlabel{$S$} [bl] at 51 160
\pinlabel{$S$} [bl] at 250 193
\pinlabel{$S$} [t] at 366 162
\pinlabel{$S$} [b] at 581 166
\endlabellist 
\centering 
\includegraphics[scale=.4]{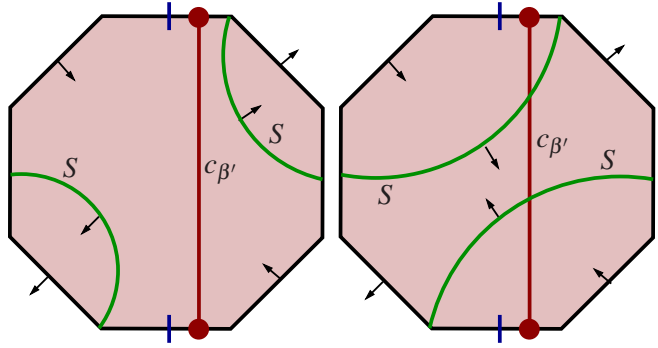}
\caption{The two possible kinds of intersection between $S$ and $D_{\beta'}$ (for the case when $S$ is homologous to $R_+$).}
\label{Fig: ATimpliesBT}
\end{figure}

\textbf{Case 1:} The arcs $S \cap D_{\beta'}$ are disjoint from $c_{\beta'}$. 

By the $e_{\beta'}$-irreducibility of $M$ and our hypotheses, we may assume that no component of $S$ is a sphere intersecting $e_{\beta'}$ one or fewer times. Let $n_S$ be the number of components of $S$ that are discs intersecting $e_{\beta'}$ exactly once. Similarly, we may assume that no component of $R(\gamma')$ is a sphere intersecting $e_{\beta'}$ one or fewer times. Recall that no component of $R(\gamma')$ containing an endpoint of $e_{\beta'}$ is a disc. We have
\[
x_{e_{\beta'}}(R_\pm) = x_{c_{\beta'}}(R_\pm) + 2
\]
and
\[
x_{e_{\beta'}}(S) = x_{c_{\beta'}}(R_\pm) + 2 - n_S
\]

If a component of $S$ is a disc intersecting $e_{\beta'}$ once, then either it is a $c_{\beta'}$-compressing disc for the component of $R_\pm$ sharing its boundary, or that component is a disc. Since $R_\pm$ is $c_{\beta'}$-taut and since no sphere intersects an edge of $e_{\beta'}$ once, that component of $R_\pm$ must be a disc intersecting $e_{\beta'}$ once, contradicting our hypotheses. Thus, $n_S = 0$. It then follows that since $R_\pm$ is $x_{c_{\beta'}}$-minimizing, $x_{e_{\beta'}}(R_\pm) \leq x_{e_{\beta'}}(S)$. Hence, $R_\pm$ is $e_{\beta'}$-taut. 

\textbf{Case 2: } The arcs $S \cap D_{\beta'}$ are not disjoint from $c_{\beta'}$.

Since the endpoints of $c_{\beta'}$ are in different components of $A(\gamma')$, we can isotope $S$ so that $\boundary S$ moves across $A(\gamma')$ and so that $S$ is made disjoint from $c_{\beta'}$. Call the resulting surface $S'$. We have $\boundary S' = \boundary R_\mp$. The intersection between $S'$ and $D_{\beta'}$ is as in Figure \ref{Fig: ATimpliesBT2}. An isotopy of $S'$ relative to $\boundary S'$ makes $S'$ disjoint from $c_{\beta'}$.

\begin{figure}[ht!] 
\labellist \small\hair 2pt 
\pinlabel {$c_{\beta'}$} [l] at 187 182
\pinlabel{$S'$} [b] at 78 179
\pinlabel{$S'$} [bl] at 254 168
\endlabellist 
\centering 
\includegraphics[scale=.4]{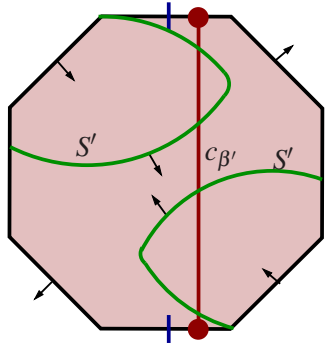}
\caption{The intersection between $S'$ and $D_{\beta'}$ (for the case when $S$ is homologous to $R_+$).}
\label{Fig: ATimpliesBT2}
\end{figure}

Let $S'' = S' \cup T(\gamma')$. Since $[S,\boundary S] = [R_\pm, \boundary R_\pm]$, we have $[S'',\boundary S''] = [R_\mp, \boundary R_\mp]$.  We note that $S''$ is $e_{\beta'}$-incompressible and that it always intersects each edge of $e_{\beta'}$ with the same sign. Consequently, by Case 1 and the fact that $R_-$ and $R_+$ have the same euler characteristic,
\[
x_{e_{\beta'}}(R_\pm) = x_{e_{\beta'}}(R_\mp) \leq x_{e_{\beta'}}(S'')= x_{e_{\beta'}}(S).
\]
Hence, $R_\pm$ is $x_{e_{\beta'}}$-minimizing and is, therefore, $e_{\beta'}$-taut.

We have proved that, in either case, $R_\pm$ is $e_{\beta'}$-taut. It is easy to show that $T(\gamma')$ is $e_{\beta'}$-taut and, therefore, that $(M,\gamma,e_{\beta'})$ is $e_{\beta'}$-taut. Consequently, $(M',\gamma',e_{\beta'} \cup c_{\beta'})$ is band-taut.
\end{proof}

We say that a sutured manifold decomposition
\[
(M,\gamma,c_\beta) \stackrel{S}{\to} (M',\gamma',c_{\beta'})
\]
is \defn{almost-taut} if $S$ is disjoint from $c_\beta$ (and so $c_\beta = c_{\beta'})$ and both $(M,\gamma,c_\beta)$ and $(M',\gamma',c_{\beta'})$ are almost taut.

\subsection{Almost taut decompositions}

To create almost taut decompositions, we recall the definition of ``Seifert-like'' homology class from the introduction: A class $y \in H_2(M,\boundary M)$ is \defn{Seifert-like} for the union $T$ of two torus components of $\boundary M$, if the projection of $y$ to the first homology of each component is non-zero. By the ``half-lives, half-dies'' theorem from algebraic topology, there are non-zero classes in the first homology of each component of $T$ that are the projections of the boundary of classes $y_1,y_2 \in H_2(M,\boundary M)$. If neither $y_1$ nor $y_2$ is Seifert-like for $T$, then summing them produces a Seifert-like homology class. Thus, if $\boundary M$ has two torus components, there is a class in $H_2(M,\boundary M)$ that is Seifert-like for their union. The next two lemmas show how to construct an almost taut decomposition, given a Seifert-like homology class.

\begin{lemma}\label{Lem: AT Decomp Exist}
Suppose that $(M,\gamma,c_\beta)$ is a $c_\beta$-almost taut sutured manifold, with $c_\beta$ an arc having both endpoints on torus components $T$ of $T(\gamma)$. Let $y$ be a Seifert-like homology class for $T$.  Then there exists a conditioned surface $S$ representing $y$ and disjoint from $c_\beta$, such that the double curve sum $S_k$ of $S$ with $k$ copies of $R(\gamma)$ is $c_\beta$-taut for any $k \geq 0$. Hence, the decomposition
\[
(M,\gamma,c_\beta) \stackrel{S_k}{\to} (M',\gamma',c_\beta)
\]
is $c_\beta$-almost taut for any $k \geq 0$.
\end{lemma}
\begin{proof}
\textbf{Claim 1:} There exists a conditioned surface $\Sigma$ representing $y$ disjoint from $c_\beta$.

Standard arguments show that there exists a conditioned surface representing $y$. Out of all such surfaces, choose one $\Sigma$ that minimizes $|\Sigma \cap c_\beta|$. By tubing together points of opposite intersection, we may assume that the geometric intersection number of $\Sigma$ with $c_\beta$ equals the absolute value of the algebraic intersection number. If this number is non-zero, we may isotope the boundary components of $\Sigma$ around a simple closed curve on one component of $T$ so as to introduce enough intersections of $\Sigma$ with $c_\beta$ of the correct sign so that $\Sigma$ and $c_\beta$ have algebraic intersection number zero. This does not change the fact that $\Sigma$ is conditioned. By tubing together points of opposite intersection, we obtain a surface contradicting our original choice of $\Sigma$.

\textbf{Claim 2:} There exists a conditioned surface $S$ representing $y$ that is disjoint from $c_\beta$ and which has the property that the double curve sum $S_k$ of $S$ with $k \geq 0$ copies of $R(\gamma)$ creates a $c_\beta$-taut surface disjoint from $c_\beta$.

We apply Theorem 2.5 of \cite{MSc3} (see page \pageref{Sch. Thm}). We apply the theorem with $R = R(\gamma)$, $C = \boundary \Sigma$, and $y = [\Sigma]$. As noted on page \pageref{Sch. Thm}, Scharlemann's theorem applies even in the absence of a sutured manifold structure, and so there is no problem with applying it in our situation. Since $R(\gamma)$ is disjoint from $c_\beta$, each of the surfaces $S_k$ is disjoint from $c_\beta$.

\textbf{Claim 3:} The manifold $(M',\gamma',c_{\beta'})$ obtained by decomposing $(M,\gamma,\beta)$, using $S_k$ from Claim 2, is $c_{\beta'}$-almost taut.

Since $S_k$ is disjoint from $c_\beta$, we have $c_{\beta'} = c_{\beta}$. The endpoints of $c_\beta$ lay in distinct components of $T(\gamma)$, so the endpoints of $c_{\beta'}$ lie in distinct components of $A(\gamma')$. The surface $R(\gamma')$ is the double curve sum of $S_k$ with $R(\gamma)$, i.e. $S_{k+1}$. Thus, $R(\gamma')$ is $c_{\beta'}$-taut. It follows easily that $(M', \gamma',c_{\beta'})$ is $c_{\beta'}$-almost taut.
\end{proof}

We now show that starting with an arc-taut sutured manifold,  converting it to an almost taut sutured manifold, applying an almost-taut decomposition, and then creating a banded sutured manifold can result in a band-taut sutured manifold.

\begin{lemma}\label{Lem: AT implies BT 2}
Suppose that $M_+$ is $\beta_+$-taut and that $y \in H_2(M,\boundary M)$ is Seifert-like for $T$. Let $S$ be a conditioned surface that represents $y$ and that gives an almost taut decomposition:
\[
M_0 \stackrel{S}{\to} (M',\gamma',c_{\beta'}).
\]
Then the banded sutured manifold $(M',\gamma',\beta')$ derived from $(M',\gamma',c_{\beta'})$ is band-taut.
\end{lemma}
\begin{proof}
Since $M_+$ is $\beta_+$-taut and since $T$ has one component in $R_-(\gamma)$ and one in $R_+(\gamma)$, $\chi(R_-(\gamma) - T) = \chi(R_+(\gamma) - T)$. Also, since $T\subset M_+$ is $\beta_+$-incompressible, no sphere in $M$ intersects $c_{\beta'}$ exactly once transversally. In $M'$, the components of $R(\gamma')$ adjacent to $T \cap M'$ each contain a copy of a component of $S$, since $S$ had boundary on both components of $T$. If one of the components of $R(\gamma')$ containing an endpoint of $e_{\beta'}$ is a disc, then some component of $S$ with boundary on $T$ must be a disk. Since $S$ is conditioned and disjoint from $c_\beta$, this implies that a component of $T$ is compressible in $M - c_\beta$ and thus in $M - \beta_+$. This contradicts the fact that $M_+$ is $\beta_+$-taut. Therefore, no component of $R(\gamma')$ containing an endpoint of $e_{\beta'}$ is a disc. Thus, by Lemma \ref{Lem: AT implies BT}, $(M',\gamma',\beta')$ is band-taut.
\end{proof}

\subsection{Parameterizing surfaces}

Suppose that $M_+$ is $\beta_+$-taut and that $y \in H_2(M,\boundary M)$ is Seifert-like for $T$. Let $S$ be a conditioned surface representing $y$ and giving an almost taut decomposition:
\[
M_0 \stackrel{S}{\to} (M',\gamma',c_\beta).
\]
Let $Q \subset M_+$ be a parameterizing surface.

\begin{lemma}\label{Lem: Arcs}
Assume that no component of $\boundary Q \cap (T - \inter{\eta}(c_\beta))$ is an inessential arc or inessential circle in $T - c_{\beta}$. Let $T'$ be a component of $T \cap M'$. The following are true:
\begin{itemize}
\item $\boundary Q \cap T'$ consists of either essential loops in $T'$ or edges joining the components of $\boundary T'$ and edges joining an endpoint of $c_\beta$ to a component of $\boundary T'$.
\item There are equal numbers of edges joining the endpoint of $c_\beta$ to the two components of $\boundary T'$.
\end{itemize}
\end{lemma}
\begin{proof}
The lemma follows immediately from the observation that on a component $T_\pm$ of $T$, each arc of $\boundary Q \cap (T_\pm - \inter{\eta}(c_\beta))$ is an essential loop. Such a loop $\sigma$ is either disjoint from $\boundary S$ or always intersects each component of $\boundary S$ with the same sign of intersection. 
\end{proof}

We observe that by Lemma \ref{Lem: T implies AT}, $M_0$ is $c_\beta$-almost taut. We do not know that $M_-$ is $\beta_-$-taut. Let $S$ be a conditioned decomposing surface giving an almost taut decomposition $M_0 \stackrel{S}{\to} (M',\gamma',c_\beta)$. Let $(M',\gamma',\beta')$ be the banded sutured manifold derived from $M'_0 = (M',\gamma',c_\beta)$. By Lemma \ref{Lem: AT implies BT 2}, $(M',\gamma',\beta')$ is band-taut. The surface $S$ also gives sutured manifold decompositions of $M_+$ and $M_-$, with $S$ disjoint from $\beta_+$ and $\beta_-$ respectively. The resulting sutured manifolds $M'_-$ and $M'_+$ can also be obtained by isotoping the endpoints of $c'_\beta$ out of $A(\gamma') \subset M'$ and into $R(\gamma) \subset M'$.  This gives us the following decompositions:
\[
\begin{array}{rclcl}
M_+ &\stackrel{S}{\to}& (M',\gamma',\beta'_+)&=&M'_+ \\
M_- & \stackrel{S}{\to} & (M',\gamma',\beta'_-)&=&M'_- \\
\end{array}\]
The arcs $\beta'_+$ and $\beta'_-$ are obtained by isotoping the arc $c_\beta \subset M'$ so that its endpoints move out of $A(\gamma)$. That is, $\beta'_+ \cup \beta'_- = e_{\beta'}$. See Figure \ref{Fig: Multidecomp} for a schematic depiction of the relationship between $M'_-$, $M'_+$, and $M'_0$.

\begin{figure}[ht!] 
\labellist \small\hair 2pt 
\pinlabel{$\phi_1$} [br] at 195 603
\pinlabel{$\phi_2$} [l] at 315 603
\pinlabel{$\phi_3$} [bl] at 415 603
\pinlabel{$\psi_1$} [l] at 88 375
\pinlabel{$\psi_2$} [l] at 311 375
\pinlabel{$\psi_3$} [l] at 554 375
\pinlabel{$\xi$} [l] at 329 182
\endlabellist 
\centering 
\includegraphics[scale=.4]{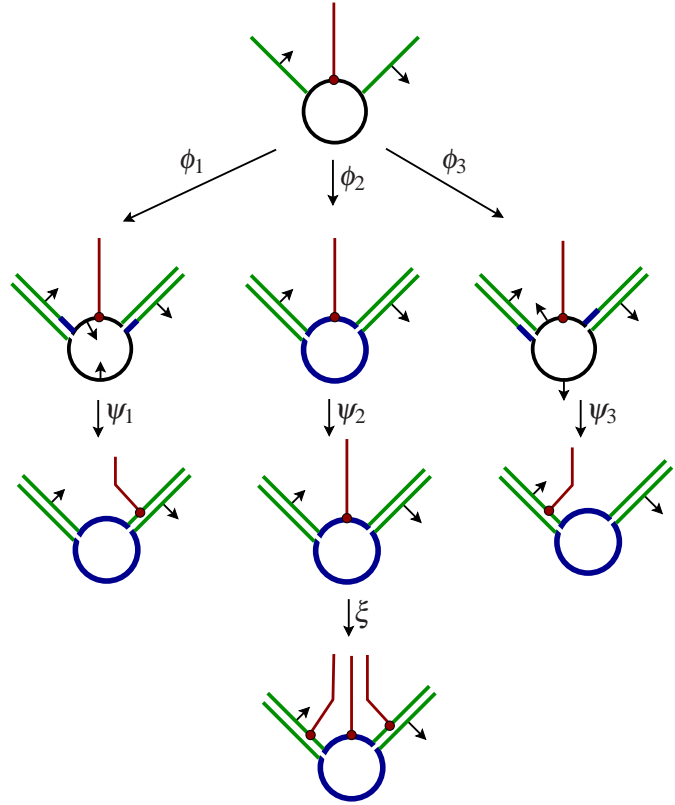}
\caption{This is a schematic depiction of the creation of $M'_0$, $M'_+$, $M'_-$, and the banded sutured manifold $(M',\gamma', \beta')$. The arrows $\phi_1$, $\phi_2$, and $\phi_3$ show the decompositions $M_+ \to M'_+$, $M_0 \to M'_0$ and $M_- \to M'_-$ respectively. The arrows $\psi_1$ and $\psi_3$, show how $M'_+$ and $M'_-$ can be obtained from $M'_0$ by an isotopy of the sutured manifold structure. The arrow $\xi$ shows how the banded sutured manifold $(M', \gamma', \beta')$ is derived from $M'_0$ and is the result of superimposing the sutured manifolds $M'_+$, $M'_0$, and $M'_-$. In all diagrams, the green lines represent the decomposing surface, blue curves represent annuli $A(\gamma)$, and the circle represents a component of $T$. }
\label{Fig: Multidecomp}
\end{figure}

If $Q_\pm$ is a parameterizing surface in $M_\pm$, then we have the decomposed surfaces $Q'_\pm \subset M'_\pm$. We assume that the ambient isotopy of $\beta_+$ to $\beta_-$ takes $Q_+$ to $Q_-$ and that $\beta_- \cap Q_+ = \beta_+ \cap Q_- = \nil$.  We say that the decomposition $M_\pm \stackrel{S}{\to} M'_\pm$ \defn{respects} $Q_+$ if $Q'_-$ and $Q'_+$ are parameterizing surfaces.

\begin{lemma}\label{Lem: AT Decomp Respect}
Suppose that $M_0$, $M_\pm$, $S$, and $Q_\pm$ are as above and that no component of $\boundary Q \cap (T - \inter{\eta}(c_\beta))$ is an inessential arc or inessential circle in $T - c_{\beta}$. Then, for large enough $k$, the decompositions
\[
M_- \stackrel{S_k}{\to} M'_- \hspace{.2in} \text{ and } \hspace{.2in} M_+ \stackrel{S_k}{\to} M'_+
\]
respect $Q$, where $S_k$ is the surface obtained by double-curve summing $S$ with $k$ copies of $R(\gamma) \subset M_0$.  If $Q'_\pm \subset M'_\pm$ are the resulting parameterizing surfaces, then $I(Q'_\pm) = I(Q)$.
\end{lemma}
\begin{proof}
Fix $k \geq 0$ and let $M'_\pm = (M',\gamma'_\pm,\beta'_\pm)$ be the result of decomposing $M_\pm$ by $S_k$. Recall that $S_k$ is disjoint from $c_\beta$. Since $\beta_-$, $\beta_+$, and $c_\beta$ are related by isotopies we may assume that $S_k$ is also disjoint from $\beta_- \cup \beta_+$. Let $M'_0 = (M'_0,\gamma'_0, c'_\beta)$ be the result of decomposing $M_0$ by $S_k$. By Lemma \ref{Lem: Arcs}, no component of $\boundary Q'_\pm \cap (T \cap M')$ is an inessential loop or arc, and no component has both endpoints on the same boundary component of $T \cap M'$. Thus, if $q$ is a disc component of $Q'_\pm$ having boundary in $R(\gamma'_\pm)$, then either $\boundary q$ is disjoint from $T \cap M'$, or $q$ is a $\beta'_\pm$ compressing disk for $T \cap M'$. If the latter happens, then $\boundary q \subset \boundary Q$. This would imply that $q$ was actually a component of $Q$ that was a $\beta_+$-compressing disk for $T$. This contradicts the fact that $M_+$ is $\beta_+$-taut. Thus, $\boundary q$ is disjoint from $T \cap M'$. This implies that $\boundary q \subset R(\gamma'_0)$.

The proof of Claim 1 of \cite[Lemma 7.5]{MSc3} shows that for large enough $k$, no component of $Q'_\pm$ is a disc with boundary in $R(\gamma'_0)$. Hence, $Q'_\pm$ is a parameterizing surface in $M'_\pm$. Claim 2 of \cite[Lemma 7.5]{MSc3} shows that $I(Q') = I(Q)$.
\end{proof}

We can now prove the main result of this paper.

\begin{theorem}\label{Thm: Main Theorem C}
Suppose that $(N,\gamma)$ is a taut sutured manifold and that $F \subset \boundary N$ is a component of genus at least 2. Let $b \subset \gamma \cap F$ be a simple closed curve such that either each component of $R(\gamma)$ adjacent to $b$ is a thrice punctured sphere or each component of $R(\gamma)$ adjacent to $b$ is a once-punctured torus.  Let $M = N[b]$ and let $\beta$ be the cocore of the 2--handle attached to $b$. Let $Q \subset N$ be a parameterizing surface. Assume that $|Q \cap b| \geq 1$ and that the intersection of $Q$ with the components of $R(\gamma)$ adjacent to $b$ contains no inessential arcs or circles. Then one of the following is true:
\begin{enumerate}
\item $Q$ has a compressing or $b$-boundary compressing disc.
\item $(N[b],\beta) = (M'_0,\beta'_0) \# (M'_1,\beta'_1)$ where $M'_1$ is a lens space and $\beta'_1$ is a core of a genus one Heegaard splitting of $M'_1$.
\item The sutured manifold $(N[b],\gamma-b)$ is $\nil$-taut. The arc $\beta$ can be properly isotoped  to be embedded on a branched surface $B(\mc{H})$ associated to a taut sutured manifold hierarchy $\mc{H}$ for $N[b]$. There is also a proper isotopy of $\beta$ in $N[b]$ to an arc disjoint from the first decomposing surface of $\mc{H}$. If $b$ is adjacent to thrice-punctured sphere components of $R(\gamma)$, that first decomposing surface can be taken to represent $\pm y$ for any given non-zero $y \in H_2(N[b],\boundary N[b])$. If $b$ is adjacent to once-punctured tori, the first decomposing surface can be taken to represent $y$ for any given homology class in $H_2(N[b], \boundary N[b])$ that is Seifert-like for the corresponding unpunctured torus components of $\boundary N[b]$.
\item 
\[
-2\chi(Q) + |Q \cap \gamma| \geq 2|Q \cap b|.
\]
\end{enumerate}
\end{theorem}
\begin{proof}
By Theorem \ref{Thm: Main Theorem B}, it suffices to prove the statement for the case when $b$ is a separating suture on a genus two surface. Convert $b$ to an arc $\beta_+$ (see Section \ref{Section: conversion}) so that we have the $\beta_+$-taut sutured manifold $M_+ = (M,\gamma - b,\beta_+)$. Let $T$ be the components of $R(\gamma - b)$ containing the endpoints of $\beta_+$. Let $y \in H_2(M,\boundary M)$ be Seifert-like for $T$. By the remarks preceding Lemma \ref{Lem: AT Decomp Exist}, such a $y$ exists. Let $Q_+ = Q$.

Isotope $\beta_+$ off itself slightly in two directions to obtain disjoint arcs $\beta_-$ and $c_\beta$. Let $M_0 = (M_0, \gamma, c_\beta)$ be the sutured manifold obtained by moving $T$ into $T(\gamma - b)$ and ignoring $\beta_+ \cup \beta_-$. Let $M_- = (M_-, \gamma, \beta_-)$ be the sutured manifold obtained by swapping the locations of the components of $T$ in $R(\gamma - b)$ and ignoring $\beta_+ \cup c_\beta$. Let $Q_-$ be the parameterizing surface in $M_-$ obtained by isotoping $Q_+$ using the isotopy taking $\beta_+$ to $\beta_-$. By a small adjustment of the isotopy, we may assume that $\beta_+ \cap Q_- = \beta_- \cap Q+ = \nil$.

Let $S$ be the surface provided by Lemma \ref{Lem: AT Decomp Exist}, so that the decomposition $M_0 \stackrel{S_k}{\to} M'_0$ is $c_\beta$-almost taut for any $k \geq 0$. Choose $k$ large enough so that the decompositions $M_- \stackrel{S_k}{\to} M'_-$ and $M_+ \stackrel{S_k}{\to} M'_+$ respect $Q$. This is possible by Lemma \ref{Lem: AT Decomp Respect}. Recall that these decompositions are disjoint from $c_\beta \cup \beta_- \cup \beta_+$, since $S_k$ is obtained by summing $S$ with copies of $R(\gamma) \subset M_0$ (and not $R(\gamma) \subset M_\pm$).  Let $Q_1 = Q'_+$ and $Q_2 = Q'_-$ be the resulting parameterizing surfaces in $M'_+$ and $M'_-$ respectively. Note that they are isotopic to each other.  By Lemma \ref{Lem: AT Decomp Respect}, we have $I(Q) = I(Q'_-) = I(Q'_+)$.  

Recall from Lemma \ref{Lem: AT implies BT 2} that the banded sutured manifold $(M',\gamma',\beta')$ derived from $(M',\gamma',c'_\beta)$ is band taut and that the components $\beta'_-$ and $\beta'_+$ of $e_\beta$ are obtained by isotopies of $c'_\beta$ in $M'_0$.  Let $e_1 = \beta'_+$ and $e_2 = \beta'_+$. By Theorem \ref{Thm: Main Theorem} one of the following occurs:
\begin{enumerate}
\item Some $Q_i$ has a compressing or $e_i$-boundary compressing disc in $(M',\gamma',e_i)$.
\item $M'$ contains an $S^2$ intersecting each of $e_1$ and $e_2$ exactly once.
\item For some $i$, $(M',e_i)$ has a connect summand that is a lens space and a core.
\item $(M',\gamma')$ is $\nil$-taut. The arc $c_\beta$ can be properly isotoped onto a branched surface $B(\mc{H}')$ associated to a taut sutured manifold hierarchy for $M'$.
\item Either $I(Q'_1) \geq 2\mu(Q'_1)$ or $I(Q'_2) \geq 2\mu(Q'_2)$.
\end{enumerate}

If (1) holds, then by Lemma \ref{Lem: compressing up}, $Q$ would have a compressing or $b$-boundary compressing disc in $(M,\gamma)$.

If (2) holds, then there is an $S^2$ in $M$ intersecting $\beta_+$ exactly once, contradicting the fact that $M_+$ is $\beta_+$-taut with the endpoints of $\beta_+$ in torus components of $R(\gamma) \subset M_+$.

If (3) holds, then since $\beta_+$ is isotopic to $c_\beta$, there is a (lens space, core) summand of $(M,\beta_+)$.

If (4) holds, then by Theorem \ref{Thm: Tautness Up}, since $S_k$ is conditioned $(M,\gamma)$ is $\nil$-taut. By construction the first decomposing surface is disjoint from the arc. Lemma \ref{Lem: Branched Surface} shows, in fact, that there is an isotopy of $c_\beta$ (rel endpoints) to lie on $\boundary M' \cup B(\mc{H}')$. There is a proper isotopy of $c_\beta$ in $M$ to lie on $S_k \cup B(\mc{H}')$. Thus, there is a branched surface $B(\mc{H})$ associated to a taut sutured manifold hierarchy $\mc{H}$ for $(M,\gamma)$ such that there is a proper isotopy of $c_\beta$ into $B(\mc{H})$.

If (5) holds, then since $I(Q) = I(Q'_1) = I(Q'_2)$ and since $\mu(Q_1) = \mu(Q_2) = \mu(Q)$, we have $-2\chi(Q) + |\boundary Q \cap \gamma| \geq 2|\boundary Q \cap b|$.
\end{proof}

\section{Tunnel Number One Knots}\label{Tunnels}
In this section we apply Theorem \ref{Thm: Main Theorem C} to the study of tunnel number one knots and links. Scharlemann and Thompson \cite[Proposition 4.2]{ST}, proved that given a tunnel for a tunnel number one knot in $S^3$, the tunnel can be slid and isotoped to be disjoint from some minimal genus Seifert surface for the knot\footnote{It is perhaps worth remarking that \cite[Proposition 4.2]{ST} depends on \cite[Lemma 4.1]{ST} whose proof relies on sutured manifold theory. Also, we should note, that Scharlemann and Thompson prove, in fact, that in many cases the tunnel can be isotoped onto a minimal genus Seifert surface. We will not address that aspect of their work.}. We generalize and extend this result in several ways:
\begin{itemize}
\item Scharlemann and Thompson's result holds for 2-component tunnel number one links in $S^3$. 
\item A similar theorem applies to all tunnel number one knots and 2-component links in any closed, orientable 3--manifold. (Of course, if a 3-manifold contains a tunnel number one knot or link, the 3-manifold has Heegaard genus less than or equal to two.)
\item A given tunnel for a tunnel number one knot or link can be properly isotoped to lie on a branched surface arising from a certain taut sutured manifold hierarchy of the knot or link exterior.
\end{itemize}
We begin with some terminology.

A link $C$ in a closed 3-manifold $M$ is a \defn{generalized unlink} if each component of $\boundary (M - \inter{\eta}(C))$ is compressible in the exterior of $C$. Suppose that $L_b \subset M$ is a knot or two-component link and that $\beta$ is an arc properly embedded in the complement of $L_b$. The arc $\beta$ is a \defn{tunnel} for $L_b$ if the exterior of $L_b \cup \beta$ is a handlebody. If $L_b$ is a two-component link this implies that $\beta$ joins the components of $L_b$. $L_b$ has \defn{tunnel number one} if it has a tunnel and is not a generalized unlink.

A \defn{generalized Seifert surface} $S$ for a knot or link $L_b$ in a closed manifold $M$ is a compact oriented surface properly embedded in $M - \inter{\eta}(L_b)$ such that $\boundary S$ consists of parallel (as oriented curves) longitudes on each component of $\boundary (M - \inter{\eta}(L_b))$. In particular, $\boundary S$ has components on each component of $\boundary (M - \inter{\eta}(L_b))$. If $\boundary S$ has a single component on each component of $\boundary (M - \inter{\eta}(L_b))$ then $S$ is a \defn{Seifert surface} for $L_b$. A generalized Seifert surface is \defn{minimal genus} if it has minimal genus among all generalized Seifert surfaces in the same homology class.

\begin{theorem}\label{Thm: Tunnel 1}
Suppose that $L_b \subset M$ has tunnel number one and that $\beta$ is a tunnel for $L_b$. Assume also that $(M - L_b, \beta)$ does not have a (lens space, core) summand. Then there exist (possibly empty) curves $\wihat{\gamma}$ on $\boundary(M - \inter{\eta}(L_b))$  such that $(M - \inter{\eta}(L_b), \wihat{\gamma})$ is a taut sutured manifold and the tunnel $\beta$ can be properly isotoped to lie on the branched surface associated to a taut sutured manifold hierarchy of $(M - \inter{\eta}(L_b), \wihat{\gamma})$. In particular, if $L_b$ has a (generalized) Seifert surface, then there exists a minimal genus (generalized) Seifert surface for $L_b$ that is disjoint from $\beta$.
\end{theorem}
\begin{proof}
Let $W = \eta(L_b \cup \beta)$ and let $N = M - \inter{W}$ be the complementary handlebody. Let $H = \boundary W$. Let $b \subset H$ be a simple closed curve that is a meridian of $\beta$, so that the exterior $N[b]$ of $L_b$ can be obtained by attaching a 2--handle to $\boundary N$ along $b$. The tunnel $\beta$ is a cocore of that 2--handle.

\textbf{Claim:} $H - b$ is incompressible in $N$.

\begin{proof}[Proof of Claim]
If $b$ is compressible in $N$, then $(W,N)$ is a reducible Heegaard splitting for $M$. Since boundary reducing a handlebody creates a handlebody, $L_b$ must be a generalized unlink. Suppose that $D$ is a compressing disc for $H- b$. If $b$ is separating, then $\boundary D$ must be an essential curve in one of the punctured torus components of $H - b$. Compressing that component using $D$ creates a compressing disc for $b$ in $N$. Thus, $b$ cannot be separating. If $b$ is non-separating then either $L_b$ is a generalized unlink or $\boundary D$ is an inessential curve in $\boundary N[b]$. In the latter case, $\boundary D$ bounds an essential disc in $W$ (obtained by banding together two copies of the disc in $W$ bounded by $b$), so once again $(W,N)$ is a reducible Heegaard splitting for $M$ and $C$ must be a generalized unlink. Thus, $H - b$ is incompressible in $N$.
\end{proof}

Let $Q \subset N$ be a pair of properly embedded non-parallel non-separating essential discs, chosen so as to intersect $b$ minimally. As a consequence of the claim, no component of $Q$ is disjoint from $b$. By the minimality of $|\boundary Q \cap b|$, each component of $Q \cap (H - b)$ is an essential arc.

If there were a $b$-boundary compressing disc $D$ for a component $Q_0$ of $Q$, then boundary compressing $Q_0$ using $D$ results in two discs, each intersecting $b$ fewer times than does $Q$ with at least one of them a compressing disc for $H$ in $N$. Thus, by the minimality of the intersection between $\boundary Q$ and $b$,  $Q$ has no $b$-boundary compressing disc. 

If $b$ is separating, choose $\wihat{\gamma} = \nil$. If $b$ is non-separating, we want to choose essential curves $\wihat{\gamma} \subset H - \inter{\eta}(b)$ with the following properties:
\begin{enumerate}
\item $\wihat{\gamma}$ consists of two essential simple closed curves that are parallel in $\boundary N[b]$ and which separate the components of $\boundary \eta(b)$.
\item Each arc component of $Q \cap (H - \inter{\eta}(b))$ is an arc intersecting $\wihat{\gamma}$ zero or one times.
\end{enumerate}

To see that this can be done, recall that the surface $H' = H - \inter{\eta}(b)$ is a twice-punctured torus and that $Q \cap H'$ is a collection of essential arcs. We describe how to find $\wihat{\gamma}$ if each component of $Q \cap H'$ joins the components of $\boundary \eta(b)$. We leave the other case as an exercise. There are at most four disjoint non-parallel essential isotopy classes $c_1, \hdots, c_4$ of arcs in $\boundary Q \cap H'$. An essential simple closed curve $\gamma_1$ can be chosen that is disjoint from representatives of two of the arcs (say $c_1$ and $c_2$) and that intersects representatives of the other two classes in a single point each. Let $\gamma_2$ be a second copy of $\gamma_1$, isotoped to be disjoint from $\gamma_1$. In $\boundary N[b]$, push a sub-arc of $\gamma_2$ along arcs of $Q - \gamma_1$ until it crosses an endpoint of $\beta$. Then $\gamma_2$ intersects $c_3$ and $c_4$ exactly once and is disjoint from $c_1$ and $c_2$. By isotoping $\wihat{\gamma} = \gamma_1 \cup \gamma_2$ in $H'$ to intersect $\boundary Q$ minimally we obtain the desired curves. See Figure \ref{Fig: Sutures} for a schematic depiction of the four isotopy classes of arcs and the sutures $\gamma_1$ and $\gamma_2$.

\begin{figure}[ht!] 
\labellist \small\hair 2pt 
\pinlabel{$\gamma_1$} [r] at 175 161
\pinlabel{$\gamma_2$} [br] at 134 187
\pinlabel{$H - b$} [bl] at 375 392
\pinlabel{$c_1$} [b] at 275 98
\pinlabel{$c_2$} [b] at 275 67
\pinlabel{$c_3$} [br] at 96 340
\pinlabel{$c_4$} [br] at 141 340
\endlabellist 
\centering 
\includegraphics[scale=.4]{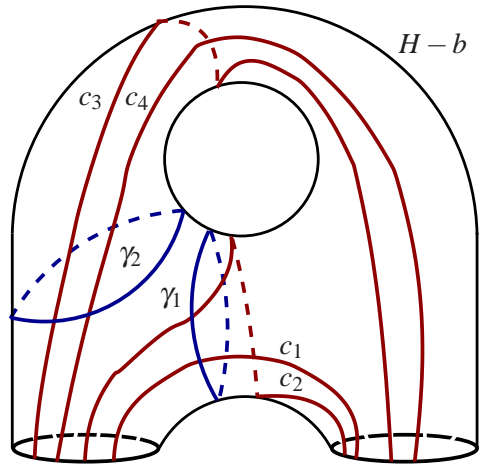}
\caption{The possible isotopy classes of arcs of $\boundary Q \cap (H -b)$ (up to homeomorphism of $H - b$) and the sutures $\gamma_1$ and $\gamma_2$ chosen to intersect those isotopy classes nicely.}
\label{Fig: Sutures}
\end{figure}

It is now easy to verify that $(N, \wihat{\gamma} \cup b)$ is a taut sutured manifold and that $|\boundary Q \cap \wihat{\gamma}| \leq |\boundary Q \cap b|$.
Since $-2\chi(Q) = -4$, it is impossible that 
\[
-2\chi(Q) + |Q \cap (\wihat{\gamma} \cup b)| \geq 2|Q \cap b|
\]

Consequently, by Theorem \ref{Thm: Main Theorem C}, $\beta$ can be isotoped to lie on a branched surface associated to a taut sutured manifold hierarchy of $(N[b],\wihat{\gamma})$. 

If $L_b$ has a (generalized) Seifert surface, choose $y \in H_2(N[b],\boundary N[b])$ to be a class represented by (generalized) Seifert surfaces for $L_b$. The first surface $S$ in the sutured manifold hierarchy constructed in the proof of Theorem \ref{Thm: Main Theorem C} is a conditioned surface representing $\pm y$ that is taut in the Thurston norm of $N[b]$ and is disjoint from $\beta$. If $\Sigma$ is a minimal genus (generalized) Seifert surface for $L_b$ representing $\pm y$, then $\Sigma$ can be isotoped to have the same boundary as $S$ and (possibly after spinning around $\boundary N[b]$ and changing orientation) is homologous to $S$ in $H_2(N[b],\boundary S)$. Since $S$ has minimal Thurston norm among all such surfaces, it is a minimal genus (generalized) Seifert surface for $L_b$ disjoint from $\beta$.
\end{proof}

Scharlemann-Thompson's result follows immediately:

\begin{corollary}[Scharlemann--Thompson]\label{ST-disjt tunnel}
Suppose that $K$ is a tunnel number one knot or link in $S^3$ with tunnel $\beta$ then $\beta$ can be isotoped to be disjoint from a minimal genus Seifert surface for $K$.
\end{corollary}


\begin{thebibliography}{99}

\bibitem[CC]{CC}
{J. Cantwell and L. Conlon}
`The sutured Thurston norm.'
{\tt arxiv/0606534}

 \bibitem[G1]{G1}
{D. Gabai} `Foliations and the topology of $3$-manifolds.'
{\em J. Differential Geom.} 18  (1983),  no. 3, 445--503.

\bibitem[G2]{G2}
{D. Gabai} `Foliations and the topology of $3$-manifolds. II.'
{\em J. Differential Geom.}  26  (1987),  no. 3, 461--478.

\bibitem[G3]{G3}
{D. Gabai} `Foliations and the topology of $3$-manifolds. III.'
{\em J. Differential Geom.}  26  (1987),  no. 3, 479--536.

\bibitem[K]{K} {E. Kalfagianni} ` Surgery n-triviality and companion tori.'
{\em  J. Knot Theory Ramifications} 13  (2004),  no. 4, 441--456.

\bibitem[L1]{L1}
{M. Lackenby}, `Surfaces, surgery and unknotting operations.' {\em Math. Ann.} 308 (1997) 4.

\bibitem[L2]{L2}
{M. Lackenby}, `Dehn surgery on knots in $3$-manifolds.' {\em J. Amer. Math. Soc.} 10 (1997) 4.

\bibitem[L3]{L3} {M. Lackenby},  `Upper bounds in the theory of unknotting operations.'
 {\em Topology}  37  (1998),  no. 1, 63--73.

\bibitem[S1]{MSc3}
{M. Scharlemann}, `Sutured manifolds and generalized Thurston norms.' {\em J. Differential Geom.} 29 (1989) 3.

\bibitem[S2]{MSc4}
{M. Scharlemann}, `Producing reducible 3--manifolds by surgery on a knot.'{\em Topology.} 29 (1990) 4.

\bibitem[ST]{ST}
{M. Scharlemann and A. Thompson}, `Unknotting tunnels and Seifert surfaces.'
{\em Proc. London Math. Soc.} 87  (2003), 3 no. 2, 523--544.

\bibitem[T1]{T1}
{S. Taylor}, `Boring split links and unknots.' Dissertation. University of California, Santa Barbara. September, 2008.


\bibitem[T2]{T3}
{S. Taylor} `Comparing 2--handle additions to a genus 2 boundary component'. arXiv:0806.1572

\end{thebibliography}
\end{document}